\documentclass[10pt,fleqn]{article}

\usepackage[T1]{fontenc}
\usepackage[latin1]{inputenc}
\usepackage{amsmath}
\usepackage{amssymb}
\usepackage{amsthm}
\usepackage{dsfont}
\usepackage{abbreviationsGeneral-ext}
\usepackage{enumerate}
\usepackage{graphicx}
\usepackage[numbers]{natbib}
\usepackage{color}
\usepackage{pifont}

\theoremstyle{plain}\newtheorem{definition}{Definition}[section]
\theoremstyle{plain}\newtheorem{lemma}[definition]{Lemma}
\theoremstyle{plain}\newtheorem{proposition}[definition]{Proposition}
\theoremstyle{definition}
\theoremstyle{plain}\newtheorem{assumption}{Assumption}
\theoremstyle{plain}\newtheorem{theorem}[definition]{Theorem}
\theoremstyle{plain}\newtheorem{corollary}[definition]{Corollary}
\theoremstyle{definition}\newtheorem{remark}[definition]{Remark}

\def\TCS{	T_{{\rm CS},n}}
\def\TMed{	T_{{\rm Med},n}}
\def\THL{	T_{{\rm HL},n}}

\def\sCS{\sigma_{\rm CS}}
\def\sMed{\sigma_{\rm Med}}
\def\sHL{\sigma_{\rm HL}}

\def\hsCS{\hat\sigma_{{\rm CS},n}}
\def\hsMed{\hat\sigma_{{\rm Med},n}}
\def\hsHL{\hat\sigma_{{\rm HL},n}}

\def\PNED{$P$NED}

\begin{document}
\title{Studentized $U$-quantile processes under dependence with applications to change-point analysis}
\author{Daniel Vogel* and Martin Wendler}
\date{}

\maketitle

\begin{center}
\footnotesize{
\noindent
Institute for Complex Systems \& Mathematical Biology, University of Aberdeen, \\
Aberdeen AB24 3UE, United Kingdom, \emph{daniel.vogel@abdn.ac.uk}  \\

\medskip
\noindent
Institut f\"ur Mathematik \& Informatik, Ernst Moritz Arndt Universit\"at Greifswald, \\
17487 Greifswald, Germany, \emph{martin.wendler@uni-greifswald.de} \\
}
\end{center}

\begin{abstract} 
Many popular robust estimators are $U$-quantiles, most notably the Hodges--Lehmann location estimator and the $Q_n$ scale estimator.
We prove a functional central limit theorem for the $U$-quantile process without any moment assumptions and under weak short-range dependence conditions.
We further devise an estimator for the long-run variance and show its consistency, from which the convergence of the studentized version of the  $U$-quantile process to a standard Brownian motion follows. 
This result can be used to construct CUSUM-type change-point tests based on $U$-quantiles, which do not rely on bootstrapping procedures.
We demonstrate this approach in detail with the example of the Hodges--Lehmann estimator for robustly detecting changes in the central location. A simulation study confirms the very good efficiency and robustness properties of the test. Two real-life data sets are analyzed.
\end{abstract}

keywords:
CUSUM test, Hodges--Lehmann estimator, Long-run variance, Median, Near epoch dependence, Robustness, Weak invariance principle.


\section{Introduction}
\label{sec:intro}

Let $X_1,\ldots,X_n$ be a (not necessarily independent) sample from some univariate distribution $F$. For a symmetric, measurable function $g:\R^2 \to \R$, the average of the $\binom{n}{2}$ values $g(X_i,X_j)$, $1 \le i < j \le n$, is called a $U$-statistic with kernel $g$. If the data are independent, this is an unbiased estimator of the quantity $E(g(X_1,X_2))$. A prominent textbook example is the scale estimator known as Gini's mean difference, which is obtained for $g(x,y) = |x-y|$.

Instead of taking the average, one may also consider the sample median of $g(X_i,X_j)$, $1 \le i < j \le n$, or more generally any sample $p$-quantile, $0 < p < 1$. Such a statistic is called a $U$-quantile. 
Several estimators that have gained popularity in robust statistics are $U$-quantiles. For instance, taking $p = 1/4$ and the above mentioned kernel $g(x,y) = |x-y|$ yields the $Q_n$ scale estimator \citep{RousseeuwCroux1993}. 
Similarly, choosing the sample median and the kernel $g(x,y) = (x+y)/2$ yields the Hodges--Lehmann estimator of location \citep{Hodges1963, Sen1963},
\be \label{eq:hl}
	\hat{h}_n = {\rm median}\left\{ (X_i + X_j)/2 \,|\, 1 \le i < j \le n \right\}.
\ee
The motivation for the present article originates in the authors' interest in robust change-point detection. Let us consider for an instant the change-point-in-location problem. Specifically, if we let $(Y_i)_{1 \le i \le n}$ be a centered stationary sequence and assume the data $(X_i)_{1\le i \le n}$ to follow the model $X_i = Y_i + \mu_i$, $1 \le i \le n$, we want to test the hypothesis
\begin{equation*} 
	H_0 : \ \mu_1= \mu_2 = \ldots = \mu_n 
\end{equation*}
against the alternative
\[
	H_1:\ \exists\, k\in\{1,\ldots,n-1\}:\ \mu_1=\ldots=\mu_k \neq \mu_{k+1} = \ldots =\mu_n.
\]
The usual CUSUM test statistic for detecting changes in the central location can be written as
\be \label{eq:cs}
	\TCS = \max_{1\le k\le n} \frac{k}{\sqrt{n}} | \bar{X}_k - \bar{X}_ n |, 
\ee
where $\bar{X}_k$ denotes the mean of the first $k$ observations. For a stationary sequence $X_k$, $k \in \Z$, satisfying suitable moment and short-range dependence conditions, $\TCS$ converges in distribution to $\sigma_{\rm CS} \sup_{t \in [0,1]} |B(t)|$, where 
\be \label{eq:cs var}
	\sCS^2 = \sum_{k = -\infty}^\infty \cov (X_0,X_k) 
\ee
is the long-run variance $\lim_{n \to \infty} \var(\bar{X}_n)$ of the mean, and $B$ denotes a Brownian bridge.
The main tool for proving the convergence of $\TCS$ is an invariance principle (or functional central limit theorem) for the partial sum process
\be \label{eq:psp}
	  \Big( \frac{1}{\sqrt{n}} \sum_{i = 0}^{[s n]} (X_i-EX_1) \Big)_{0 \le s \le 1} 
		=\Big( \frac{[ n s ]}{\sqrt{n}} ( \bar{X}_{1:[ns]} - EX_1) \Big)_{0 \le s \le 1},
\ee
which one may also view as a partial mean process. The first objective of the present paper is to establish a functional limit theorem under short-range dependence for the $U$-quantile process, i.e., the process obtained from the right-hand side of (\ref{eq:psp}) by replacing the sample mean by a $U$-quantile and $EX_1$ by the corresponding population value (Theorem \ref{theo:1}). The second main theoretical contribution is to propose and establish the consistency of an estimator for the long-run variance term that appears in the limit process (Theorem \ref{theo:2}).
These results can be used to devise a CUSUM-type change-point test for location based on the Hodges--Lehmann estimator, which is expected to have a much higher robustness against heavy tails than the classical CUSUM test while retaining essentially the same efficiency under normality, as it is known that the Hodges--Lehmann estimator has an asymptotic efficiency of 95\% with respect to the mean at normality \citep[e.g.][]{choudhury:serfling:1988}. Similarly, the classical approach to the change-in-scale detection problem is a CUSUM-type test statistic, where the mean is replaced by the sample variance. This goes back to \citet{inclan:tiao:1994}, and has been extended to broader settings by several authors \citep{gombay:horvath:huskova:1996, lee:park:2001, wied:arnold:bissantz:ziggel:2012}. This test suffers even more so from the vulnerability to outliers and heavy tails. Our results can also be used to devise an alternative test for changes in the variability based on the highly robust $Q_n$ scale estimator.

The outline of the paper is as follows. The limit theorems for general $U$-quantiles are given in Section \ref{sec:limit}, with the proofs being deferred to the Appendix. In Section \ref{sec:hl}, we investigate the application of the results to the problem of change-in-location detection by means of the Hodges--Lehmann estimator. In Section \ref{sec:sim}, we analyze power and finite-sample properties of this test and compare it to the classical CUSUM test and a similar test based on the median by means of numerical simulations. The simulation results confirm that the good efficiency and robustness properties of the Hodges--Lehmann estimator translate into similar properties of the test. The application of the test is demonstrated at two data examples in Section \ref{sec:data}.


\section{Limit theorems for $U$-quantiles under dependence}
\label{sec:limit}

Let $(X_i)_{i\in\Z}$ be a strictly stationary sequence of random variables. 
The empirical $p$-$U$-quantile can be written as the generalized inverse $U^{-1}_n(p)$  of the empirical $U$-distribution function 
\[
	U_n(t)  = \frac{2}{n(n-1)}\sum_{1\leq i<j\leq n}\ind{g\left(X_i,X_j\right)\leq t}.
\] 
To allow smoothed estimators of the generalized distribution function as well, we replace $\ind{g\left(x,y\right)\leq t}$ by a more general function $h(x,y,t)$.

\begin{definition} \label{def:1}
We call a nonnegative, bounded, measurable function $h:\R\times\R\times\R\rightarrow\R$ which is symmetric in the first two arguments and non-decreasing in the third argument a $U$-quantile kernel function. For fixed $t\in\R$, we call
\begin{equation*}
	U_n(t) = \frac{2}{n(n-1)}\sum_{1\leq i<j\leq n}h\left(X_i,X_j,t\right)
\end{equation*}
the $U$-statistic with kernel $h\left(\cdot,\cdot,t\right)$ and the process $\left(U_n(t)\right)_{t\in\R}$ the empirical $U$-distribution function. 
We define the population $U$-distribution function as $U(t) = E\left[h\left(X,Y,t\right)\right]$, where $X$, $Y$ are independent with the same distribution as $X_0$.
Furthermore, $U^{-1}(p) = \inf\{t|U(t)\geq p\}$ is called the $p$-$U$-quantile and $U_n^{-1}(p) = \inf\{t|U_n(t)\geq p\}$ the empirical $p$-$U$-quantile.
\end{definition}

To study the empirical $U$-distribution function, we need a functional version of the Hoeff\-ding decomposition \citep{hoeffding:1948}. We write $U_n(t)$ as
\begin{equation*}
	U_n(t)=U(t)+\frac{2}{n}\sum_{i=1}^{n}h_{1}\left(X_{i},t\right)+\frac{2}{n\left(n-1\right)}\sum_{1\leq i<j\leq n}h_{2}\left(X_{i},X_{j},t\right)
\end{equation*}
where
\begin{align} \label{eq:hoeffding}
	h_1(x,t) & =Eh(x,X_0,t)-U(t), \\ 
	h_2(x,y,t) &=h(x,y,t) - h_1(x,t) - h_1(y,t) - U(t). \nonumber
\end{align}
$U$-quantiles can be analyzed using a generalized Bahadur representation. \citet{Bahadur1966} showed that the empirical quantile can be approximated by a linear transform of the empirical distribution function. This was generalized by \citet{Geertsema1970} to $U$-quantiles of independent data. The rate of convergence was improved by \citet{choudhury:serfling:1988}, \citet{dehling:denker:philipp:1987} and \citet{Arcones1996} later. A generalized Bahadur representation for $U$-quantiles of dependent data was recently established by \citet{Wendler2011,Wendler2012}.

Concerning the serial dependence structure of the process $(X_i)_{i\in\Z}$, we assume it to be near epoch dependent in probability (\PNED) on an absolutely regular process.
For two $\sigma$-fields $\mathcal{A},\mathcal{B}\subset\mathcal{F}$ on the probability space $(\Omega,\mathcal{F},P)$, the absolute regularity coefficient $\beta (\mathcal{A},\mathcal{B}) = E [\sup_{A\in\mathcal{A}}\left| P (A |\mathcal{B}) - P (A) \right|]$ is a measure of dependence of $\calA$ and $\calB$. Let $(Z_i)_{i\in\Z}$ be a stationary process. The absolute regularity coefficients of $(Z_i)_{i\in\Z}$ are given by
\begin{equation*}
	\beta_{k} = \beta\left(\sigma(\ldots,Z_{-1},Z_0),\sigma(Z_{k},Z_{k+1},\ldots)\right), \qquad k \in \N.
\end{equation*}
The process $(Z_i)_{i\in\Z}$ is called absolutely regular if $\beta_k\rightarrow0$ as $k\rightarrow\infty.$
We will not study absolutely regular processes themselves, as important classes of time-series like linear processes are not covered.  Instead, we study processes which are near epoch dependent on absolutely regular processes. 
\begin{definition} Let $\left((X_i,Z_i)\right)_{i\in\Z}$ be a stationary process.
\begin{enumerate}
\item 
We say that $(X_i)_{i\in\Z}$ is $L_p$ near epoch dependent, $p \ge 1$, on the process $(Z_i)_{i\in\Z}$ with approximation constants $(a_{l,p})_{l\in\N}$ if $\lim_{l \rightarrow \infty} a_{l,p} = 0$ and
\begin{equation*}
	\left(E \left| X_0 - E (X_0 | \sigma(Z_{-l}, \ldots, Z_l) ) \right|^p\right)^{\frac{1}{p}} \leq a_{l,p},  \qquad l \in \{ 0, 1,2, \ldots \}.
\end{equation*}
\item 
We say that $\left(X_i\right)_{i\in\Z}$ is near epoch dependent in probability (\PNED) on the process $(Z_i)_{i\in\Z}$ with approximation constants $(a_l)_{l\in\N}$ if $a_l \to 0$ as $l \to \infty$ and there is a sequence of functions $f_l:\R^{2l+1}\rightarrow \R$ and a non-increasing function $\phi: (0,\infty)\rightarrow(0,\infty)$ such that
\begin{equation*}
	P\left(|X_0-f_l(Z_{-l},\ldots,Z_l)|>\epsilon\right)\leq a_l\phi(\epsilon)
\end{equation*}
for all $l\in\N$ and $\epsilon>0$.
\end{enumerate}
\end{definition}
Near epoch dependent processes are also called approximating functionals \citep[e.g.][]{borovkova:burton:dehling:2001}.
This class of short-range dependent processes includes all time series models relevant in econometrics, like ARMA-processes and GARCH-processes \citep[e.g.][]{Hansen1991}, and furthermore also covers expanding dynamical systems, where the sequence $X_{n+1}=T(X_n)$ is deterministic apart from the initial value $X_0$ \citep[see e.g.][]{hofbauer:keller:1982}.

We prefer to use near epoch dependence in probability (\PNED) instead of the usual $L_2$ near epoch dependence since it does not necessitate the existence of any moments. We consider quantile-based estimators, a decisive advantage of them being their moment-freeness, and we do not want to limit the scope of our results in this respect by implicitly introducing moment assumptions in the short-range dependence conditions.
The concept of \PNED\ used here was introduced by \citet{dehling:vogel:wendler:wied:2015:vs4}. Similar concepts that embody the idea of approximating $(X_i)_{i\in\Z}$ in a probability sense rather than an $L_p$ sense can be found under the name of $S$-mixing in \citet{berkes:hoermann:schauer:2009} and under the name of $L_0$-approximability in \citet[][Chapter 6]{poetscher:prucha:1997}. If $(X_i)_{i\in\Z}$ is near epoch dependent in probability on the process $(Z_i)_{i\in\Z}$, we can represent $X_n$ almost surely as $X_n=f_\infty((Z_{n+l})_{l\in\Z})$.
We will require the \PNED\ approximation constants $a_l$ and the absolute regularity coefficients $\beta_k$ to fulfill certain rate conditions.
\begin{assumption} \label{ass:ned}
The sequence $(X_i)_{i\in\Z}$ is \PNED\ on an absolutely regular sequence $(Z_i)_{i\in\Z}$ such that $a_l\phi(l^{-6}) = O(l^{-6})$ as $l \to \infty$ and $\sum_{k=1}^\infty k\beta_k < \infty$.
\end{assumption}

So far, the $U$-statistic kernel $g$ is completely arbitrary. In proofs for weakly dependent data, the dependent random variables are approximated by independent random variables. In order to control the error induced by this approximation, we require some form of continuity condition on $h$ with respect to the marginal distribution of the process. 
\begin{assumption} \label{ass:variation}
Let $0 < p < 1$ and $h:\R\times\R\times\R\rightarrow\R$ be a bounded kernel function such that for a constant $L$ and for all $t$ in a neighborhood of $U^{-1}(p)$ and all $\epsilon >0$
\begin{equation*}
 E\left[\sup_{\substack{x,y:\\ \left\|(x,y)-(X,Y)\right\|\leq \epsilon}} \left|h\left(x,y,t\right)-h\left(X,Y,t\right)\right|^2\right]\leq L\epsilon,
\end{equation*}
where $X$, $Y$ are independent with the same distribution as $X_0$ and $\left\|(x_1,x_2)\right\|=(x_1^2+x_2^2)^{1/2}$ denotes the Euclidean norm.
\end{assumption}
This condition holds for all Lipschitz continuous kernel functions $h$. If Lipschitz continuity does not hold, as it is the case for kernels of the type $h(x,y,t) = \ind{g\left(x,y\right)\leq t}$, we need some regularity conditions on the distribution of $X_0$, cf.\ Remark \ref{rem:1} below.

Since we consider sample quantiles, we further require that the $U$-distribution function $U$ behaves regularly at $U^{-1}(p)$. Let $u(t) = U'(t)$ denote the derivative of the $U$-distribution function. 
\begin{assumption} \label{ass:smooth}
Let $U(t) = E\left[h\left(X,Y,t\right)\right]$ be differentiable in a neighborhood of $U^{-1}(p)\in\R$ with $u\left(U^{-1}(p)\right)>0$ and
\begin{equation} \label{eq:smooth}
\left|U(t)-p-u\left(U^{-1}(p)\right)\left(t-U^{-1}(p)\right)\right|
= o\left(\left|t-U^{-1}(p)\right|^{3/2}\right)\ \ \ \text{as}\ \ t\rightarrow U^{-1}(p).
\end{equation}
\end{assumption}

We are now ready to state the first of our two main results.

\begin{theorem} \label{theo:1} 
Under Assumptions \ref{ass:ned}, \ref{ass:variation}, and \ref{ass:smooth}, we have for the $U$-quantile process that 
\begin{equation*}
	\left(\frac{[ns]}{\sqrt{n}}\left(U_{[ns]}^{-1}(p)-U^{-1}(p)\right)\right)_{s\in[0,1]} \ \cid \ \sigma_p W
\end{equation*}
in the Skorokhod space $D[0,1]$, where $W$ is a standard Brownian motion and 
\begin{equation} \label{eq:lrv}
	\sigma^2_p = \frac{4}{u^2(U^{-1}(p))}\sum_{r=-\infty}^\infty \cov\left(h_1(X_0,U^{-1}(p)),h_1(X_r,U^{-1}(p))\right).
\end{equation}
\end{theorem}

%
%

Unless the distribution of the whole process $(X_i)_{i\in\Z}$ is fully specified, the long-run variance $\sigma_p^2$ is unknown. For statistical applications, it is therefore desirable to have an estimate of $\sigma_p^2$. The estimator we propose below is obtained by replacing all unknown quantities in the right-hand side of (\ref{eq:lrv}) by their empirical versions. We restrict our attention to the original situation where $h$ takes on the form $h(x,y,t) = \ind{g(x,y) \le t}$. This allows to directly apply usual kernel density estimation to the $U$-statistic density $u$. Let
\be \label{eq:density}
	\hat{u}_n(t) = \frac{2}{n(n-1)d_n}\sum_{1\leq i<j\leq n}K\left(\frac{g(X_i,X_j)-t}{d_n}\right), 
\ee
where $K$ is a density kernel and $d_n$ a bandwidth which fulfill the following conditions. 
\begin{assumption} \label{ass:density}
The function $K$ is symmetric around 0, Lipschitz continuous with bounded support and bounded variation, and it integrates to 1. The bandwidth $d_n$ satisfies $d_n\rightarrow0$ and $n d_n^{8/3}\rightarrow\infty$ as $n \to \infty$.
\end{assumption}
Furthermore, we need an empirical version of $h_1$ from (\ref{eq:hoeffding}). Let 
\[
	\hat{h}_1(x,t) = \frac{1}{n}\sum_{i=1}^n h(x,X_i,t)-\frac{1}{n^2}\sum_{i,j=1}^nh(X_i,X_j,t),
\]
and consider the sample autocovariance of $(\hat{h}_1(X_i,t))_{1 \le i \le n}$ for lag $r$, i.e.,
\begin{equation*}
	\hat{\rho}(r,t) =\frac{1}{n}\sum_{i=1}^{n-r}\hat{h}_1(X_i,t)\,\hat{h}_1(X_{i+r},t).
\end{equation*}
We estimate the infinite-sum part in (\ref{eq:lrv}) by a heteroscedasticity and autocorrelation consistent (HAC) kernel estimator, and define
\[
	\hat{\sigma}^2_{p,n}
	= \frac{4}{\hat{u}_n^2(U_n^{-1}(p))}
	\, \sum_{r=-(n-1)}^{n-1} W(r/b_n) \, \hat{\rho}(r,U_n^{-1}(p)),
\]
where $W$ and $b_n$ fulfill the following conditions.
\begin{assumption}\label{ass:hac} The function $W:[0,\infty)\rightarrow[0,1)$ is continuous at 0 and at all but a finite number of points. Furthermore, $|W|$ is dominated by a non-increasing,
integrable function and
$\int_0^\infty\left|\int_0^\infty W(t)\cos(xt)dt\right|dx<\infty$.
The bandwidth $b_n$ satisfies $b_n\rightarrow \infty$ and $b_n/\sqrt{n}\rightarrow0$ as $n\rightarrow\infty$.
\end{assumption}
Assumption \ref{ass:hac} mainly coincides with Assumption 1 of \citet{dejong:2000}. It is satisfied by a large class of kernels, including the Bartlett kernel $W(t) = (1 - |t|) \ind{|t| \le 1}$.
Finally, we need a continuity condition similar to Assumption \ref{ass:variation} also for the kernel $g$.
\begin{assumption} \label{ass:variation2}
There is a constant $L$ such that for all $\epsilon>0$
\begin{equation*}\label{line8}
		E\left(\sup_{\substack{x,y:\\ \left\|(x,y)-(X,Y)\right\|\leq \epsilon}}\left|g\left(x,y\right)-g\left(X,Y\right)\right|\right)^2\leq L\epsilon,
\end{equation*}
where $X$, $Y$ are independent with the same distribution as $X_0$.
\end{assumption}
Conditions of this type (including Assumption \ref{ass:variation} above) are also called variation conditions and were first introduced by \citet{denker:keller:1986}.  They are mild regularity conditions which we usually find to be fulfilled for kernels and data distributions that are of interest for statistical applications. Specific conditions on the distribution $F$ implied by Assumptions \ref{ass:variation} and \ref{ass:variation2} in case of the Hodges--Lehmann estimator are discussed in Remark \ref{rem:1}.

We have the following consistency result for the long-run variance estimator. 
\begin{theorem} \label{theo:2} 
Under Assumptions 1 to 6 we have $\hat\sigma^2_{p,n} \cip \sigma^2_p$ 
as $n\rightarrow\infty$.
\end{theorem}
The following result is an immediate corollary of Theorems \ref{theo:1} and \ref{theo:2}. Part (A) follows by Slutsky's lemma, and part (B) by a further application of the continuous mapping theorem. 
\begin{corollary} \label{cor:1} Under Assumptions 1 to 6, we have 
\begin{enumerate}[(A)]
\item
$ \displaystyle 
	\left(\frac{[ns]}{\sqrt{n}\, \hat{\sigma}_{p,n}} \left(U_{[ns]}^{-1}(p)-U^{-1}(p)\right)\right)_{s\in[0,1]} \ \cid \ W$ \\[1.5ex]
in the Skorokhod space $D[0,1]$, where $W$ is a standard Brownian motion, and
\item
$ \displaystyle
	\max_{2 \le k\le n} \, \frac{k}{\sqrt{n}\, \hat{\sigma}_{p,n}} \, | U^{-1}_k(p) - U^{-1}_n(p) | \cid \sup_{0\leq s \leq 1} |B(s)|$, \\[1.5ex]
 where  $B(s) = W(s)-s W(1)$, $0 \le s \le 1$, is a standard Brownian bridge.
\end{enumerate}
\end{corollary}
We refer to the process in Corollary \ref{cor:1} (A) as the studentized $U$-quantile process.


\section{Robust detection of changes in the central location}
\label{sec:hl} 

We return to the question of change-point detection as outlined in the introduction. The practical implementation of the CUSUM test, cf.~(\ref{eq:cs}), requires the estimation of the long-run variance $\sCS^2$, cf.~(\ref{eq:lrv}), which is usually accomplished by a kernel estimator of the form
\be \label{eq:cs var est}
	\hsCS^2 = \sum_{k=-(n-1)}^{n-1} W(k/b_n) \bigg\{ \frac{1}{n} \sum_{i =1}^{n-|k|} (X_i - \bar{X}_n)(X_{i+|k|}-\bar{X}_n) \bigg\},
\ee
where $W$ and $b_n$ are as in Assumption \ref{ass:hac} \citep[see e.g.][]{aue:horvath:2013}.
The CUSUM test is known to be inefficient under heavy tails and prone to outliers.
It is interesting to note that, although outliers tend to increase the test statistic $\TCS$, the general effect outliers have on the test is not a size distortion, but rather a loss of power: the test statistic is divided by the estimate $\hsCS$, which is even more strongly increased by outliers. 
An intuitive approach to a robust, less outlier-sensitive change-point detection is to replace the sample mean in (\ref{eq:cs}) by an alternative location estimator. We will pursue this approach in the following and examine the median and the Hodges--Lehmann estimator $\hat{h}_n$, cf.~(\ref{eq:hl}), as potential alternatives. 

The problem of change-point-in-location detection is a classic one and well studied, see, e.g., the monograph by \citet{csorgo:horvath:1997}. Articles considering the problem under dependence include among others \citet{Andrews1993}, \citet{kokoszka:leipus:1998}, 
\citet{horvath:kokoszka:steinebach:1999}  and \citet{horvath:steinebach:2000}. The literature on robust analysis of the change-point problem is comparably limited. There are approaches, e.g., based on ranks \citep[e.g.][]{Huskova1997,antoch:huskova:janic:ledwina:2008}, $M$-estimators \citep[e.g.][]{Huskova1996} and $U$-statistics \citep[e.g.][]{Gombay2000,gombay:horvath:2002}. All of these consider independent sequences. Recently, \citet{huskova:marusiakova:2012} considered robust change-point procedures for $\alpha$-mixing sequences. See \citet{Huskova2013} for a recent overview on robust change-point analysis. \citet{Hoyland1965} and \citet{dehling:fried:2012} consider two-sample tests based on the two-sample Hodges--Lehmann estimator for independent and dependent data, respectively, which may provide the basis for robust change-point tests based on the two-sample Hodges-Lehmann estimator.

When replacing the mean in (\ref{eq:cs}), one possibility is the median, presumably the simplest robust location estimator, leading to the test statistic
\be \label{eq:med}
	\TMed = \max_{1\le k\le n} \frac{k}{\sqrt{n}} | \hat{m}_k - \hat{m}_n |, 
\ee
where $\hat{m}_k$ denotes the median of $X_1, \ldots, X_k$. Under the null hypothesis of no change and under appropriate regularity conditions (which include no moment conditions, but smoothness conditions on the distribution $F$ of $X_1$), $\TMed$ converges in distribution to  $\sMed \sup_{t \in [0,1]} |B(t)|$ with 
\be \label{eq:med var}
	\sMed^2 = \frac{1}{f(m)^2} \sum_{k = -\infty}^\infty \cov \left(\ind{X_0\le m}, \ind{X_k\le m}\right) 	
\ee
where $m = F^{-1}(1/2)$ denotes the median of the distribution $F$ and $f$ its density. This convergence result as well as the consistency of the 
long-run variance estimator
\be \label{eq:med var est}
	\hsMed^2 =  \frac{1}{\hat{f}_n(\hat{m}_n)^2} 
	\sum_{k=-(n-1)}^{n-1} W(k/b_n) \Bigg\{ \frac{1}{n} \sum_{i =1}^{n-|k|} (\ind{X_i\le \hat{m}_n} - 1/2)(\ind{X_{i+|k|}\le \hat{m}_n} - 1/2) \Bigg\}
\ee
with a suitable kernel density estimator
\be
	\hat{f}_n(x) = \frac{1}{n\, d_n} \sum_{k =1}^n K \{ (X_k - x)/d_n \} 
\ee
can be shown by similar techniques as Theorems \ref{theo:1} and \ref{theo:2}. 
However, this robustification is paid by a substantial loss in efficiency at normality. The median is known to possess an asymptotic relative efficiency of $\pi/2 = 64\%$ with respect to the mean for independent Gaussian observations. Hence we propose to use the Hodges--Lehmann estimator,
which is also highly robust but possesses an asymptotic relative efficiency of $3/\pi = 95\%$ with respect to the mean at normality. This leads to the test statistic 
\be \label{eq:hl test}
	\THL = \max_{1\le k\le n} \frac{k}{\sqrt{n}} | \hat{h}_k - \hat{h}_n |. 
\ee
It should be noted that \citet{Hodges1963} actually consider the variant 
$\tilde{h}_n = {\rm median}\{(X_i + X_j)/2, X_k \,|\, 1 \le i < j \le n, \, 1 \le k \le n\}$. 
Since $\tilde{h}_n$ and $\hat{h}_n$ behave very similarly and are asymptotically equivalent, we stick to the variant $\hat{h}_n$, to which the $U$-quantile theory applies directly.

For stationary and short-range dependent sequences, this test statistic $\THL$ converges in distribution to $\sHL \sup_{t \in [0,1]} |B(t)|$, where $B$ is a Brownian bridge and
\be \label{eq:hl var}
	\sHL^2 = \frac{4}{u(h)^2} \sum_{k = -\infty}^\infty E \left( \psi(X_0) \psi(X_k) \right).
\ee
Here, $u$ is the density of the distribution of $(X+Y)/2$ for $X, Y \sim F$ independent, $h$ its median, and $\psi(x) = P((x+Y)/2 \le h)-1/2$ for $Y \sim F$.
Implementing the long-run variance estimation technique for $U$-quantiles described in Section \ref{sec:limit}, one obtains
\be \label{eq:hl var est}
	\hsHL^2 =  \frac{4}{\hat{u}(\hat{h}_n)^2} 
	\sum_{k=-(n-1)}^{n-1} W(k/b_n) \Bigg\{ \frac{1}{n} \sum_{i =1}^{n-|k|} \hat\psi_n(X_i) \hat\psi_n(X_{i+|k|}) \Bigg\},
\ee
where $\hat{u}_n$ is given by (\ref{eq:density}) for the kernel $g(x,y)=(x+y)/2$, and 
	$\hat\psi_n(x) =  n^{-1} \sum_{j=1}^n (\ind{(x+X_j)/2 \le \hat{h}_n} - 1/2)$.
The asymptotic behavior of the studentized test statistic $\THL/ \hsHL$ is given by Corollary \ref{cor:1} (B) and is summarized in the following corollary.
\begin{corollary} \label{cor:2}
Let $(X_i)_{i\in\Z}$ be a stationary sequence with marginal distribution $F$ which satisfies Assumption \ref{ass:ned}. Let $F$ be such that Assumptions \ref{ass:variation} and \ref{ass:smooth} are fulfilled for the kernel $h(x,y,t) = \ind{(x+y)/2 \le t}$. If further Assumptions \ref{ass:density} and \ref{ass:hac} are satisfied, then
$\THL/ \hsHL$ converges in distribution to $\sup_{t \in [0,1]} |B(t)|$, where $B$ is as before a Brownian bridge.
\end{corollary}
\begin{remark} \label{rem:1}
It is desirable to translate Assumptions \ref{ass:variation} and \ref{ass:smooth} for this kernel $h$ into a set of easy-to-verify conditions on $F$. It is sufficient that $F$ possesses a Lebesgue density $f$ which satisfies the following three conditions:
\begin{enumerate}[(A)]
\item
$f$ is cadlag on $\R$,
\item \label{num:2}
$\displaystyle \sup_{-\infty < s < t < \infty} \left| \frac{f(t-) - f(s)}{t-s} \right| < \infty$, and
\item \label{num:3}
the support of $f$ (i.e.\ the closure of $\{ x | f(x) > 0 \}$) is a connected set or $f$ is symmetric around some point in $\R$.
\end{enumerate}
Assumption \ref{ass:variation2} is met for all distributions $F$ since the kernel $g(x,y) = (x+y)/2$ is Lipschitz continuous. The function $f$ being both a density and cadlag (right-continuous and left-hand side limits) on $\R$ implies that $f$ is bounded, hence $F$ is Lipschitz continuous, from which Assumption \ref{ass:variation} follows. Concerning Assumption \ref{ass:smooth}, $f$ being cadlag also implies that $f$ has at most countably many discontinuity points, which together with (\ref{num:2}) implies (\ref{eq:smooth}). Condition (\ref{num:2}) if fulfilled, e.g., if $f$ possesses a right-hand side derivative $f'$ everywhere, and $f'$ is cadlag. Note that the $U$-density $u$ in this case is up to re-scaling the convolution of $f$ with itself. Finally, either of the conditions of (\ref{num:3}) ensures that the $U$-density $u$ is non-zero in a neighborhood around its median.

\end{remark}


\section{Simulations}
\label{sec:sim}

We present Monte Carlo simulation results to investigate the size and power properties of the three tests proposed in the previous section. We have carried out simulations for several sample sizes, but the results presented are for $n = 240$ only. This sample size is large enough for the asymptotics to provide sensible approximations, and the picture is the same at other sample sizes as far as the comparison of the tests is concerned. Throughout, we use 1000 replications. We consider two different scenarios concerning the characteristics of the marginal distribution of the data generating process,
\begin{enumerate}[(A)]
\item symmetric data distributions, \label{scen.A}
\item \label{scen.B} skewed data distributions. The set-up will be such that a change in variance occurs along with the change in location.
\end{enumerate}
%
%
%
%
%
%
%
In {\bf scenario (\ref{scen.A})}, we generate data from the following general one-change-point model:
\[
 X_i = Y_i + \mu \ind{i > \lfloor \theta n \rfloor}, \qquad i = 1,\ldots,n,
\]
where $Y_i$, $i \in \Z$, is a stationary sequence, $\mu$ the jump height, and $\theta$ a jump location parameter. We use the following three marginal distributions for the process $(Y_i)_{i\in\Z}$: normal, $t_3$, and $t_1$. The $t_\nu$ distribution with parameter $\nu > 0$ has the density
	$f_\nu(x) =  \sqrt{\nu} B(\nu/2,1/2) ( 1 + x^2/\nu)^{-(\nu+1)/2}$,
where $B$ is the beta function. In order to make the jump sizes better comparable among the different marginal distributions, we scale the $t_\nu$ distribution such that the median (of the distribution) of $|Y_1|$ is the same as in the normal case, i.e., we multiply the $t_\nu$ realizations by $\gamma_\nu = z_{3/4}/t_{\nu;3/4}$, where $z_\alpha$ and $t_{\nu;\alpha}$ denote the $\alpha$-quantiles of the normal distribution and the $t_\nu$ distribution, respectively. 
Concerning the serial dependence of the sequences, we consider two cases:
\begin{enumerate}[({A}.1)]
\item independence, i.e., the $Y_i$, $i \in \Z$, are i.i.d.
\item AR(1), i.e., $Y_i = \gamma_\nu F_\nu^{-1} \{ \Phi(Z_i/\sqrt{1-\phi^2}) \}$, where the $Z_i$ fulfill the auto-regressive equation $Z_i = \phi Z_{i-1} + \epsilon_i$ with $\epsilon_i \sim N(0,1)$ i.i.d.\ and $\phi = 0.4$. Here, $F_\nu^{-1}$ denotes the quantile function of the $t_\nu$ distribution and $\Phi$ the cdf of the standard normal distribution.
Thus $(Y_i)_{i\in\Z}$ is a marginal transformation of a Gaussian AR process. It is again scaled such that median$(F_{|Y_1|}) = z_{3/4}$.
\end{enumerate}
	
In the independence case (A.1), the values of the long-run variances are
\[
	\sCS^2 = \var(Y_1) = 
	\begin{cases}
		1 													& \qquad \mbox{ for the normal distribution, } \\
		\gamma_\nu^2 \, \nu /(\nu-2) 				& \qquad \mbox{ for the $t_\nu$ distribution } (\nu > 2), \\
	\end{cases}
\]\[
	\sHL^2 = 
	\begin{cases}
		\pi/3 													& \qquad \mbox{ for the normal distribution, } \\
		\gamma_\nu^2  /( 3 u_\nu^2) 		& \qquad \mbox{ for the $t_\nu$ distribution,} \\
	\end{cases}
\]
where $u_\nu = 2\int_{-\infty}^\infty f_\nu^2(x) d x$. Explicit expressions are available for the convolution of a $t_\nu$-density with itself for odd integer $\nu$, see \citet{nadarajah:dey:2005}. We obtain $\sHL^2 = (2\pi/5)^2$ for $\nu = 3$ and $\sHL^2 = \pi^2/3$ for $\nu = 1$. Furthermore
\[
	\sMed^2 = 
	\begin{cases}
		\pi/2 																& \qquad \mbox{ for the normal distribution, } \\
		 \gamma_\nu^2  / \{ 4 f_\nu(0)^2\} 		& \qquad \mbox{ for the $t_\nu$ distribution.} \\
	\end{cases}
\]
In the AR(1) scenario (A.2), we have
\[
	\sCS^2 = (1+\phi)/(1-\phi)
\]
for normality. 
As for the $t_\nu$ distribution, we are not aware of an explicit expression for the moment correlation of a bivariate distribution characterized by a Gaussian copula and $t_\nu$ margins. We have furthermore
\[
	\sHL^2 = 
	\begin{cases}
		\frac{\pi}{3} + 4 \sum_{k=1}^\infty \arcsin \left( \frac{\phi^k}{2} \right) & \quad \mbox{ for the normal distribution, } \\
		(\gamma_\nu / u_\nu)^2 	\left\{ \frac{1}{3} + \frac{4}{\pi}	\sum_{k=1}^\infty \arcsin \left( \frac{\phi^k}{2}\right) \right\} & \quad \mbox{ for the $t_\nu$ distribution,} \\
	\end{cases}
\]\[
	\sMed^2 = 
	\begin{cases}
		\frac{\pi}{2} + 2 \sum_{k=1}^\infty \arcsin (\phi^k ) & \quad \mbox{ for the normal distribution, } \\
		\{ \gamma_\nu / f_\nu(0)\}^2 	\left\{ \frac{1}{4} + \frac{1}{\pi}	\sum_{k=1}^\infty \arcsin (\phi^k ) \right\} & \quad \mbox{ for the $t_\nu$ distribution.} \\
	\end{cases}
\]
We can thus study the behavior of the test statistics under the null and their respective long-run variance estimators individually. In the tables below, we distinguish three ways of dealing with the long-run variance. We use the 
\be \label{eq:bamm}
\begin{minipage}{0.8\textwidth}
\bit
\item[\ding{172}]
known values, 
\item[\ding{173}]
marginal variance estimates, and
\item[\ding{174}]
full long-run variance estimates.
\eit
\end{minipage}
\ee
Full long-run variance estimation adjusts for possible serial dependence, i.e., we use the estimators $\hsCS^2$, $\hsHL^2$ and $\hsMed^2$ as given by (\ref{eq:cs var est}), (\ref{eq:hl var est}) and (\ref{eq:med var est}), respectively. 
Marginal long-run variance estimation means we assume independence and only include the summand that corresponds to $k =0$ in the sums in (\ref{eq:cs var est}), (\ref{eq:hl var est}) and (\ref{eq:med var est}).
We take the following choices for bandwidths and kernels, 
\be \label{eq:choices}
	K(t) = \frac{3}{4} (1- t^2) \varind{[-1,1]}(t), 
	\ \ 
	W(t) = (1- t^2)^2 \varind{[-1,1]}(t),
  \  \
	d_n = I_n\, n^{-1/3}, 
	\ \ 
	b_n = 2 n^{1/3},
\ee
where $I_n$ denotes the sample interquartile range of the data points the kernel density estimator is applied to. The kernel $K$ above is known as Epanechnikov kernel, and $W$ as quartic kernel. 
The two kernels serve different purposes, $K$ is used for density estimation and must be scaled such that it integrates to 1, while $W$ is used for autocorrelation-consistent variance estimation and must be scaled such that $W(0) = 1$. These choices are ultimately arbitrary, but they have been shown to perform well in simulations over a wide range of scenarios. 
The results generally differ very little with respect to the choice of the kernel. 
We compute for each sample the test statistics $\TCS$, $\THL$, $\TMed$, divide them by the square root of the corresponding long-run variance estimate and count how often the thus adjusted test statistic exceeds the critical value 1.358, which is the 95\% quantile of the limiting distribution. Although  based on highly robust estimators, the test statistics $\THL$ and $\TMed$ are susceptible to outliers. Problems can arise when several extreme values occur at the beginning of the sequence. In order to improve the robustness of the tests, we apply an ad-hoc fix and simply exclude the first 10 values from the sequences of successive estimates before taking the maximum.

\begin{table}[t]
\small
\caption{{\it Test size.} 
Rejection frequencies (\%) at the asymptotic 5\% significance level of the CUSUM test (\ref{eq:cs}), the Hodges--Lehmann test (\ref{eq:hl test}), and the median test (\ref{eq:med}) under no change. Marginal distributions: normal, $t_3$, and $t_1$; dependence scenarios: independence and AR(1) with parameter $\phi= 0.4$; sample size $n = 240$; 1000 runs. Long-run variance estimation: \ding{172}, \ding{173}, \ding{174}, cf.~(\ref{eq:bamm}). }
\renewcommand{\arraystretch}{1.0}
\medskip
\centering
\begin{tabular}{cc||rrr|rrr|rrr}
\multicolumn{2}{r@{\quad}||}{ test: } & \multicolumn{3}{c|}{ CUSUM } & \multicolumn{3}{c|}{ Hodges--Lehmann } & \multicolumn{3}{c}{ median }  \\
\multicolumn{2}{r@{\quad}||}{ long-run variance: } 
	& \quad\ding{172}\ {} & \quad\ding{173}\ {} & \quad\ding{174}\ {}
	& \quad\ding{172}\ {} & \quad\ding{173}\ {} & \quad\ding{174}\ {} 
	& \quad\ding{172}\ {} & \quad\ding{173}\ {} & \quad\ding{174}\ {} \\[1.0ex]
\hline \hline 
independent & normal     & 4 & 4 & 3 & 3 & 3 & 3 & 9 & 8 & 8 \\ 
data 				&	$t_3$      & 5 & 3 & 2 & 4 & 4 & 2 & 8 & 7 & 8 \\ 
						& $t_1$        &   &  1 &  1 &  7 &  6 &  5 & 10 &  8 & 10 \\[1.0ex] 						
\hline 
AR(1)				& normal     &  4 & 31 &  3 &  4 & 30 &  3 &  8 & 27 &  8 \\ 
$\phi = 0.4$&	$t_3$        &   & 26 &  3 &  4 & 30 &  3 &  9 & 29 & 10 \\ 
						& $t_1$        &   &  6 &  0 &  8 & 34 &  5 &  9 & 26 &  8 \\ 
\end{tabular} \label{tab:1}
\end{table}
\begin{figure}[ht]
\centering
\includegraphics[width=0.9\textwidth]{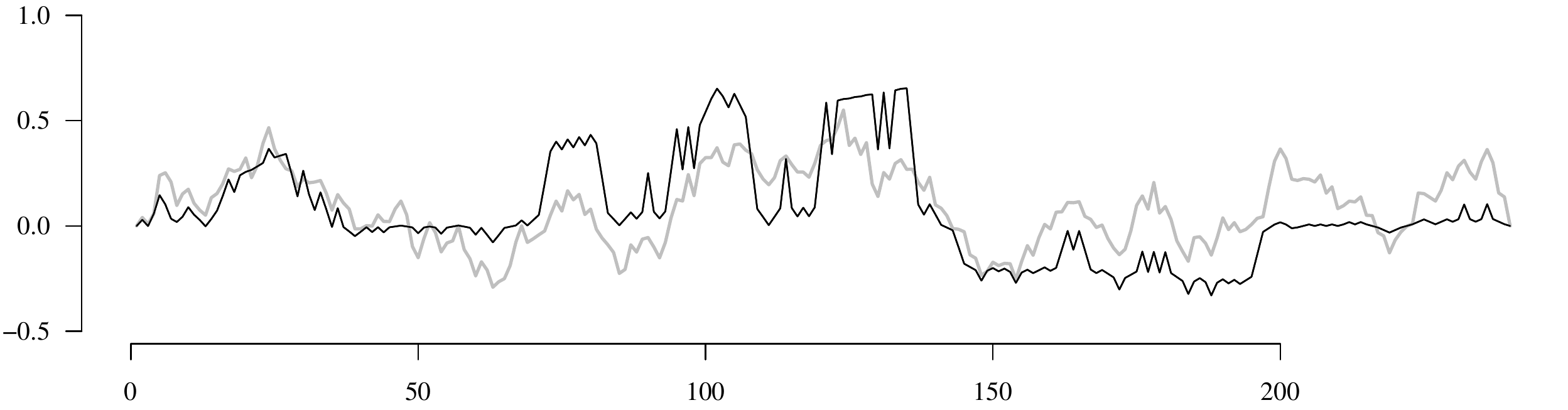}
\caption{A typical trajectory of the change-point process of the median test $(k n^{-1/2} \hat\sigma_{{\rm Med},n}^{-1} (\hat{m}_k - \hat{m}_n))_{k = 1,\ldots,n}$ (black) and the analogue for the CUSUM test (gray) for $n = 240$ independent standard normal observations.} \label{fig:path}
\end{figure}
%
%
%
%
\paragraph{\it Analysis of size.} The results for the size of the tests are summarized in Table~\ref{tab:1}. We observe the following:
\begin{enumerate}[(1)]
\item The CUSUM test and the test based on the Hodges--Lehmann estimator (referred to as Hodges--Lehmann test in the following) keep the nominal size of 5\% for the normal and the $t_3$ distribution under independence as well as dependence, but appear to be slightly conservative. 
\item
The median test shows a substantial size distortion in all situations, also when the test statistic is adjusted by the true variance. It persists also for considerably larger $n$. This size distortion in line with results reported in Shao and Zhang (2012), for which the self-normalization approach proposed by the authors does not provide a remedy either. 
This behavior may be described as a discretization problem in finite samples: the $\hat{m}_k$, $k = 1,\ldots, n$, take on only a small number of distinct values. The resulting paths differ strongly from the paths of a Brownian bridge also for large samples ($n > 1000$), and the distribution of the supremum very slowly approaches its limit. In principle, this discretization applies also to the Hodges--Lehmann test, but to a negligible extent. For $n = 240$, the Hodges--Lehmann estimator is the median of about 30,000 (in case of continuous models) distinct values. A typical trajectory of the median change-point process of a standard normal i.i.d.\ sample with estimated long-run variance is depicted in Figure~\ref{fig:path}.
Due to the size distortion, the median test is excluded from the detailed power considerations in the following. In summary, it has a power comparable to that of the Hodges--Lehmann test when not corrected for size, but when corrected for size, it has in all situations, a much lower efficiency.
\item
As expected, the marginal variance estimation fails in the AR(1) case. Ignoring the serial dependence leads to clearly wrong results.
\end{enumerate}
\begin{table}[ht]
\small
\caption{{\it Test power under independence.} 
Rejection frequencies (\%) at the asymptotic 5\% significance level of the CUSUM test (\ref{eq:cs}), the Hodges--Lehmann test (\ref{eq:hl test}), and the median test (\ref{eq:med}) under one-jump alternatives for independent errors. Data distributions: normal, $t_3$, and $t_1$; Sample size $n = 240$; 1000 runs.
Long-run variance estimation: \ding{172}, \ding{173}, \ding{174}, cf.~(\ref{eq:bamm}). } 
\renewcommand{\arraystretch}{1.0}
\medskip
\centering
\begin{tabular}{ccc||rrr|rrr}
\multicolumn{3}{r@{\quad}||}{ test: } & \multicolumn{3}{c|}{ CUSUM } & \multicolumn{3}{c}{ Hodges--Lehmann }  \\
\multicolumn{3}{r@{\quad}||}{ long-run variance: } 
	& \quad\ding{172}\ {} & \quad\ding{173}\ {} & \quad\ding{174}\ {}
	& \quad\ding{172}\ {} & \quad\ding{173}\ {} & \quad\ding{174}\ {} \\[1.0ex]
	 jump:   &  location & height 	& \multicolumn{3}{c|}{  } & \multicolumn{3}{c}{  } \\[1ex]
\hline \hline 
normal     &   1/2 	&   1/4 	&  38 	&  37 &  29 	& 38 	&  36 &  29 \\ 
 data           		&    		&    1/2	&  94 &  93 	& 86 	&  93 &   92 &   84 \\ 
           &     		&    1 		&  100 	&  100&  100 	& 100 & 100 &  100 \\ 
           &   3/4 	&   1/4 	&  19 	&  19 &  12 	& 19 	&  18 &  11 \\ 
           &  			&   1/2 	&  75 	&  75 &  49 	& 74 	&  71 &  46 \\ 
           &    		&   1 		& 100 	& 100 & 100 	& 100 & 100 &  98 \\[1ex] 						
\hline 
$t_3$ data &   1/2 	&   1/4 	&  16 	&  19 &  14 	&  31 &  29 &  22 \\ 
           &     		&    1/2 	&   57 	&  63 &   51 	&  86 &  85 &   75 \\ 
           &     		&    1 		&  100 	&  98 &   96 	& 100 & 100 &  100 \\ 
           &   3/4 	&   1/4 	&   8 	&   9 &   6 	&  18 &  15 &   9 \\ 
           &    		&   1/2 	&  33 	&  38 &  22 	&  65 &  62 &  39 \\ 
           &    		&   1 		&  95 	&  93 &  81 	& 100 & 100 &  95 \\[1ex] 						
\hline 
$t_1$ data &   1/2 	&   1/4 	&      &   10 &   10 	&  31 &  25 &  18 \\ 
           &    		&    1/2 	&      &    2 &    2 	&  81 &  74 &   58 \\ 
           &    		&    1 		&      &    9 &    6 	& 100 & 100 &   99 \\ 
           &   3/4 	&   1/4 	&      &   10 &   10 	&  18 &  13 &  10 \\ 
           &    		&   1/2 	&      &   2 	&   10 	&  60 &  52 &  28 \\ 
           &    		&   1 		&      &   3 	&   2 	&  99 &  97 &  73 \\ 						
\end{tabular} \label{tab:2}
\end{table}
\begin{table}[ht]
\small
\caption{{\it Test power for AR(1) with $\phi = 0.4$.} 
Rejection frequencies (\%) at the asymptotic 5\% significance level of the CUSUM test (\ref{eq:cs}), the Hodges--Lehmann test (\ref{eq:hl test}), and the median test (\ref{eq:med}) under one-jump alternatives with AR(1) errors. Marginal data distributions: normal, $t_3$, and $t_1$; sample size $n = 240$; 1000 runs.
Long-run variance estimation: \ding{172}, \ding{174}, cf.~(\ref{eq:bamm}). } 
\renewcommand{\arraystretch}{1.0}
\medskip
\centering
\begin{tabular}{ccc||rr|rr}
		%
\multicolumn{3}{r@{\quad}||}{ test: } & \multicolumn{2}{c|}{ CUSUM } & \multicolumn{2}{c}{ Hodges--Lehmann }  \\
\multicolumn{3}{r@{\quad}||}{ long-run variance: } 
	& \qquad\ding{172}\ {} & \qquad\ding{174}\ {}
	& \qquad\ding{172}\ {} & \qquad\ding{174}\ {} \\[1.0ex]
	 jump:   &  location & height 	& \multicolumn{2}{c|}{  } & \multicolumn{2}{c}{  } \\[1ex]		
\hline \hline 
normal     &   1/2 	&   1/4 	&  17 &    14 &  18 	&   13 \\ 
 data      &    		&    1/2 	&  57 &    47 &   56 	&   45 \\ 
           &     		&    1 		&  99 &    98 &   99 	&   98 \\[1ex] 
           &   3/4 	&   1/4 	&   9 &     6 &   9 	&   6  \\ 
           &   			&   1/2 	&  35 &    20 &  34 	&   17 \\ 
           &    		&   1 		&  94 &    77 &  94 	&   71 \\[1ex] 						
\hline  
$t_3$ data &   1/2 	&   1/4 	&      &   7 	&  14 	&   10 \\ 
           &    		&    1/2 	&      &   24 &  51 	&   37 \\ 
           &     		&    1 		&      &   79 &  99 	&   92 \\[1ex] 
           &   3/4 	&   1/4 	&      &    3 &   8 	&    6 \\ 
           &    		&   1/2 	&      &   10 &  29 	&   14 \\ 
           &    		&   1 		&      &   44 &  90 	&   57 \\[1ex] 						
\hline  
$t_1$ data &   1/2 	&   1/4 	&      &   10 &  16 	&    11 \\ 
           &    		&    1/2 	&      &    2 &   46 	&    28 \\ 
           &    		&    1 		&      &    3 &   97 	&    79 \\[1ex] 
           &   3/4 	&   1/4 	&      &   10 &  13 	&     7 \\ 
           &    		&   1/2 	&      &   10 &  25 	&    12 \\ 
           &    		&   1 		&      &   10 &  76 	&    29 \\ 					
\end{tabular} \label{tab:3}
\end{table}

\paragraph{\it Analysis of power.} In Table~\ref{tab:2}, power results in the independence scenario (A.1) for several alternatives are given. We consider jump heights $\mu = 1/4, 1/2, 1$ and jump locations $\theta = 1/4, 1/2, 3/4$. The results for $\theta = 1/4$ are similar to those for $\theta = 3/4$ and not reported here. We find from Table~\ref{tab:2}:
\begin{enumerate}[(1)]
\item
The CUSUM test has, as expected, no power at the $t_1$ distribution. Since the second moments of the $t_1$ distribution are infinite, neither the CUSUM test statistic nor the long-run variance estimator $\hsCS^2$ converges.  
\item
The CUSUM test and the Hodges--Lehmann test perform very similarly at the normal distribution, with minor advantages for the CUSUM test. The Hodges--Lehmann test is clearly more efficient at the $t_3$ distribution and has still good power at the $t_1$ distribution. 
\item
By comparing the power of the tests with known variance and with estimated variance, we find that, although a change in location generally increases the variance estimate, thus decreasing the power of the test, this effect is rather small in case of the marginal variance estimation, cf.\ columns \ding{173}.  The marginal variance estimation provides an upper bound on what might be possibly gained by a sophisticated, data adaptive selection of the bandwidth $b_n$. 
\end{enumerate} 
In Table~\ref{tab:3}, power results for the AR(1) scenario (A.2) are given with the same choices of the parameters $\mu$ and $\theta$ and the same marginal distributions as in Table~\ref{tab:2}. All tests have a lower power in the presence of positive autocorrelations, but the conclusions concerning the rankings of the tests are the same as in the independent case.

The data generating process in {\bf scenario (\ref{scen.B})} is similar to that in scenario (\ref{scen.A}). The data follow the one-change-point model
\[
	X_i = 
	\begin{cases}
		Y_i, 													& \qquad 1 \le i \le \lfloor \theta n \rfloor, \\
		Y_i/\lambda_2,						 		& \qquad \lfloor \theta n \rfloor  + 1 \le i \le n, \\
	\end{cases}
\]
where the $Y_i$, $i \in \Z$, are exponentially distributed with parameter $\lambda = 1$. Instead of a change in the central location of a symmetric distribution we consider now a change in the parameter $\lambda$ of the exponential distribution, which implies a change in the variability along with the change in the location. The set-up is inspired by the river Elbe discharge data example in Section \ref{sec:data}, which, as a referee has pointed out, exhibits such features. To give an impression how the tests perform in such a situation, we only consider independent observations and and a change in the middle of the observed period, i.e., $\theta = 1/2$. The kernel and bandwidth choices for the long run variance estimation are as in scenario (A). We use as before $n = 240$ observations and 1000 repetitions. For i.i.d.\ sequences $Y_i$, $i \in \Z$, of $\mathrm{Exp}(1)$, we have 
	$E(Y_1) = 1$, $m = {\rm median} (Y_1) = \log(2)$, and $\sCS^2 = \sMed^2 = \var(Y_1) = 1$,
Furthermore, the population value of the Hodges--Lehmann estimator $h$ is the solution to $2(1+2 h) = e^{2h}$, and $\sHL^2 = \{3-(2h-1)^2\}/(2 h)^2$. The empirical rejection probabilities of the three tests under scenario (\ref{scen.B}) for several values of $\lambda_2$ are given in Table~\ref{tab:4}. We find that, as in scenario (\ref{scen.A}) under normality, the CUSUM test and the Hodges--Lehmann test behave similarly, and appear to equally well detect changes in the location if a change in variance occurs at the same time. Here we include also power results for the median test and note that it has a similar power as the other tests but clearly exceeds of the nominal 5\% level under the null. 

\begin{table}[t]
\small
\caption{{\it Exponential distribution.} 
Rejection frequencies (\%) at the asymptotic 5\% significance level of the CUSUM test (\ref{eq:cs}), the Hodges--Lehmann test (\ref{eq:hl test}), and the median test (\ref{eq:med}) at independent exponentially distributed observations; The parameter $\lambda$ changes from 1 in the 1st half to $\lambda_2$ in the 2nd half; sample size $n = 240$; 1000 runs. Long-run variance estimation: \ding{172}, \ding{173}, \ding{174}, cf.~(\ref{eq:bamm}). }
\renewcommand{\arraystretch}{1.0}
\medskip
\centering
\begin{tabular}{c||rrr|rrr|rrr}
\multicolumn{1}{r@{\quad}||}{ test: } & \multicolumn{3}{c|}{ CUSUM } & \multicolumn{3}{c|}{ Hodges--Lehmann } & \multicolumn{3}{c}{ median }  \\
\multicolumn{1}{r@{\quad}||}{ long-run variance: } 
	& \quad\ding{172}\ {} & \quad\ding{173}\ {} & \quad\ding{174}\ {}
	& \quad\ding{172}\ {} & \quad\ding{173}\ {} & \quad\ding{174}\ {} 
	& \quad\ding{172}\ {} & \quad\ding{173}\ {} & \quad\ding{174}\ {} \\[1.0ex]
mean 2nd half $1/\lambda_2$ & & & & & & & &  & \\ 
\hline
\hline
	1          		&   4 &  3 &  3 						&  5 &  5 &  4 						& 11 	&  8 &  9 \\ 
	4/5       		& 		& 29 & 22 						&    & 32 & 27 						&    	& 32 & 32 \\ 
 2/3       		&  		& 78 & 66 						&    & 77 & 70 						&    	& 64 & 59 \\ 
	1/2        		&     & 100 &  98 					&     & 100 &  99 				&     &  97 &  92 \\ 
 1/3       		&     & 100 & 100 					&     & 100 & 100 				&    	& 100 & 100 \\ 
\end{tabular} \label{tab:4}
\end{table}


\section{Data examples}
\label{sec:data}
We consider two data sets, both from hydrology: the maximum annual discharge of the river Elbe at Dresden and the annual rainfall in Argentina. 

The first data set has recently been analyzed by \citet{sharipov:tewes:wendler:2014}. It consists of the annual maximum discharge of the river Elbe at Dresden, Germany, in the years 1851 to 2012. The time series is depicted in Figure~\ref{fig:dd}.  There appears to be shift in the time series around the year 1900, with the annual maximum discharge being lower on average afterwards. Industrialization and infrastructural development at the end of the 19th century led to a significant discharge of industrial sewage in to the river Elbe upstream from Dresden, making the river less prone to freezing in winter, resulting in lower spring floods. 
\begin{figure}[t]
	\centering
	\includegraphics[width=\textwidth]{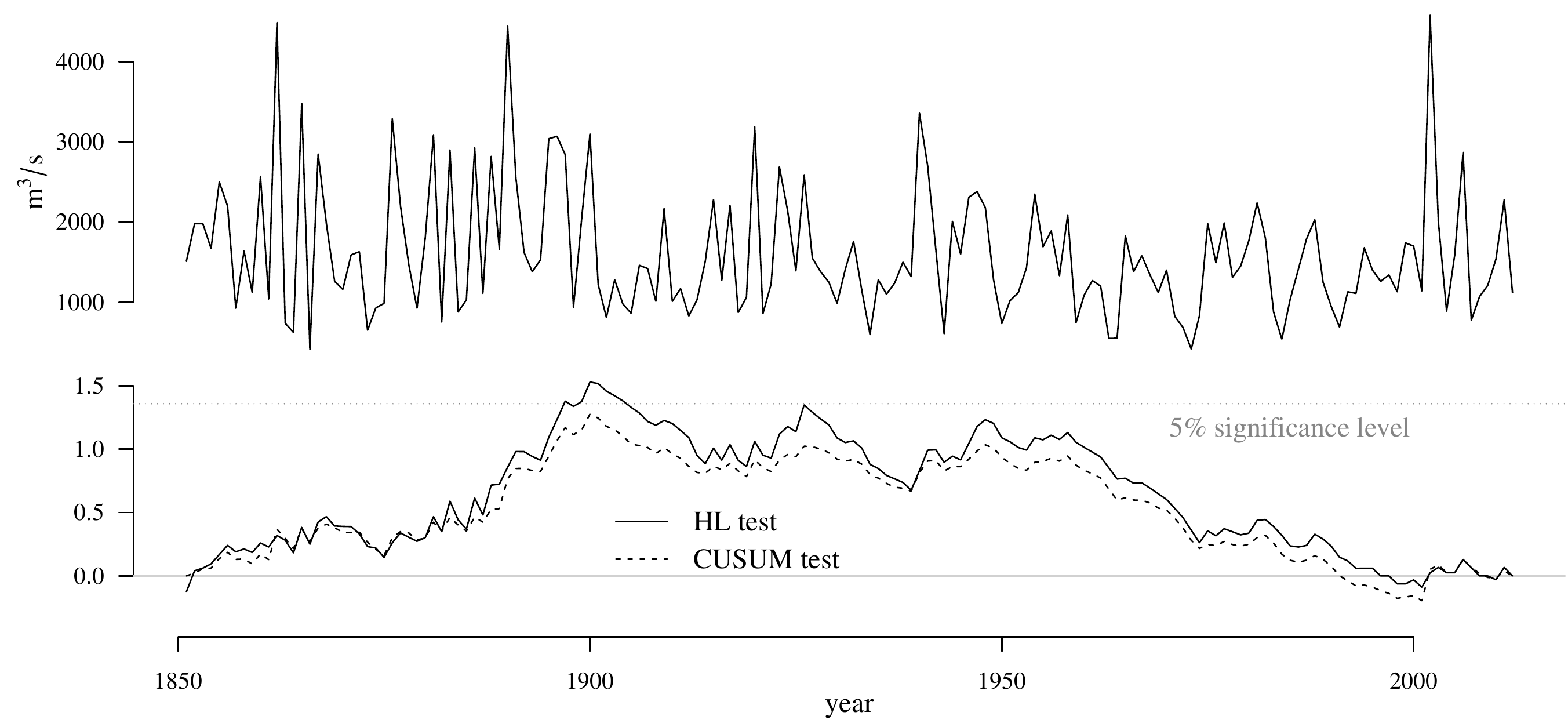}
	\caption{
		Top: Maximum yearly discharge (in cubic meter per second) of the river Elbe at Dresden from 1851 until 2012 ($n = 162$). 
		Bottom: Change-point processes 
		$(k n^{-1/2} \hsHL^{-1} (\hat{h}_k - \hat{h}_n))_{k = 1,\ldots,n}$ (solid line) and 
		$( k n^{-1/2} \hsCS^{-1}(\bar{X}_k - \bar{X}_n)_{k = 1,\ldots,n}$ (dashed line). 
	} 
	\label{fig:dd}
\end{figure}
The series is clearly non-normal, cf.\ Figure~\ref{fig:qq} (left). It exhibits a heavy upper tail, with three extreme floods in 1862, 1890, and 2002. Extreme events tend to dominate any moment based analysis such as the CUSUM test, potentially obscuring the visible change in the central location. 
Applying the CUSUM and the Hodges--Lehmann test with the choices for $K$, $W$, $d_n$, and $b_n$ as in the simulations section, cf.\ (\ref{eq:choices}), we observe that both change-point processes, i.e., $(k n^{-1/2} \hat\sigma_{{\rm HL},n}^{-1} (\hat{h}_k - \hat{h}_n))_{k = 1,\ldots,n}$ and 
$( k n^{-1/2} \hsCS^{-1}(\bar{X}_k - \bar{X}_n))_{k = 1,\ldots,n}$, which are depicted in the lower plot of Figure~\ref{fig:dd}, look similar and take their maxima at 1900. However, the test decision at the 5\% significance level is different: contrary to the Hodges--Lehmann test, the CUSUM test does not reject the hypothesis of no change. 
However, with $b_n = 2 n^{1/3}$ the HAC bandwidth is chosen rather large, while a look at the sample autocorrelations suggests that it is legitimate to treat the observations as independent.
When excluding the autocovariances from the long-run variance estimation, both tests consistently reject the null hypothesis. The heavy tail renders the CUSUM test inefficient, making the test outcome at the 5\% level sensitive to the choice of tuning parameters, whereas the Hodges--Lehmann test clearly detects the change, regardless of the choice of $b_n$. 
With the average yearly maximum discharge, the variability of the time series decreases. The simulation results of scenario (\ref{scen.B}) in the previous section suggest that the CUSUM as well as the Hodges--Lehmann test are valid in such a situation.

The second example is the Argentina rainfall data that has previously been analyzed in a change-point context by \citet{wu:woodroofe:mentz:2001} and \citet{shao:zhang:2010}. Also in this example, there is evidence (a dam built from 1952 to 1962) that supports the assumption of a change in the central location. The series is depicted in Figure~\ref{fig:ar}. The normal quantile plot (Figure~\ref{fig:qq}, right) reveals a fair agreement with normality, and in fact the Hodges--Lehmann and the CUSUM test behave similarly with both processes attaining their minima at 1955. Both reject the null hypothesis at the 5\% level for $b_n = n^{1/3}$, cf.~Figure~\ref{fig:ar}.
\begin{figure}[t]
	\centering
	\includegraphics[width=\textwidth]{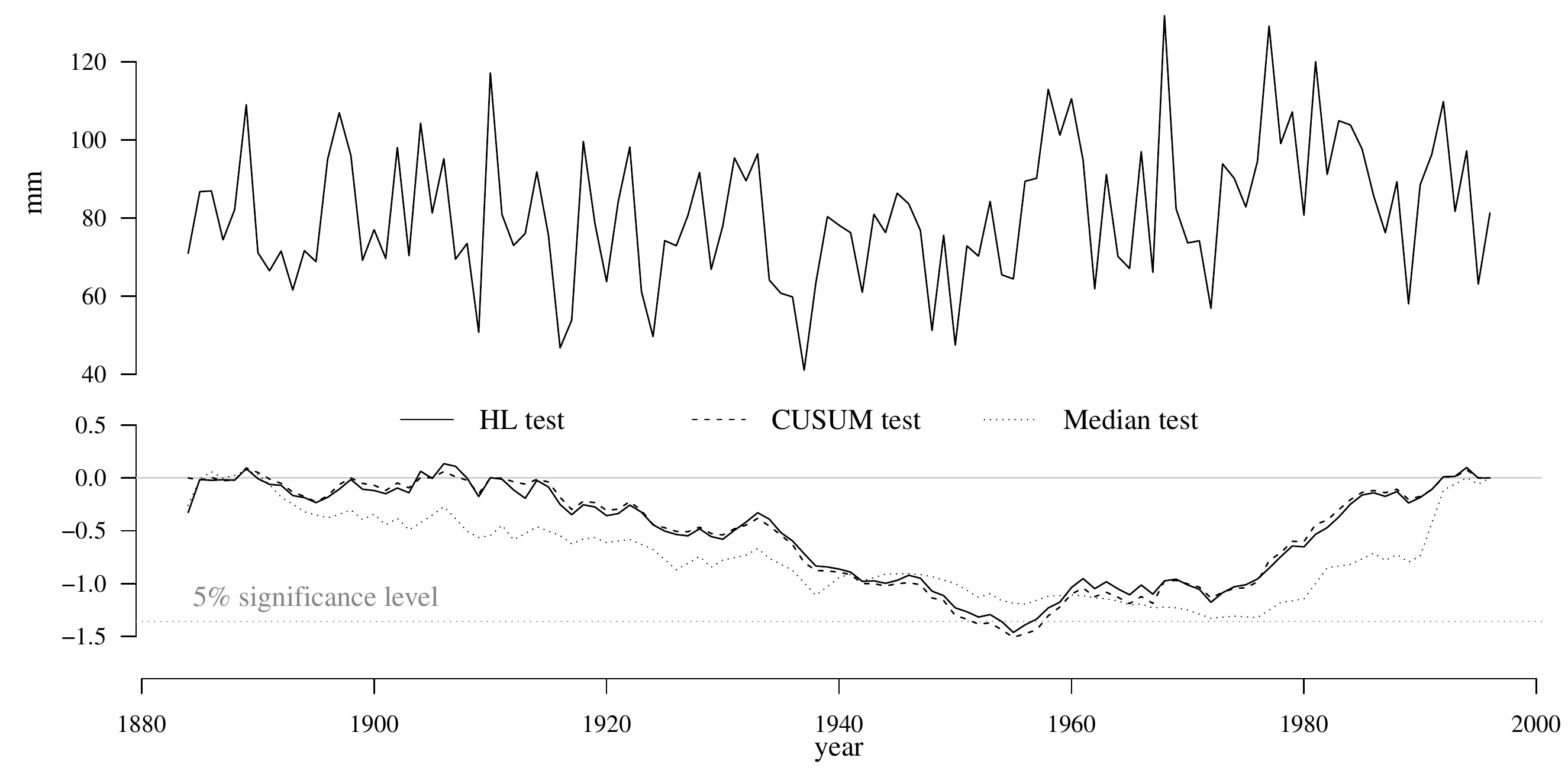}
	\caption{
		Top: Yearly rainfall (in millimeters) in Argentina from 1884 until 1996 ($n = 113$). 
		Bottom: Change-point processes 
		$(k n^{-1/2} \hsHL^{-1} (\hat{h}_k - \hat{h}_n) )_{k = 1,\ldots,n}$ (solid line),
		$( k n^{-1/2} \hsCS^{-1} (\bar{X}_k - \bar{X}_n))_{k = 1,\ldots,n}$ (dashed line), and
		$(k n^{-1/2} \hsMed^{-1} (\hat{m}_k - \hat{m}_n) )_{k = 1,\ldots,n}$ (dotted line).
		}
	\label{fig:ar}
\end{figure}
Following the analysis of \citet{shao:zhang:2010}, we also apply the median-based test to this data example (dotted line in Figure~\ref{fig:ar}). The median test does not reject the hypothesis of no change. This is in apparent contradiction to the analysis by \citet{shao:zhang:2010}, who report a p-value of less than 0.001. The authors apply a self-normalized version of the test, but since self-normalization tends to decrease the power, this is unlikely to be responsible for the different results. We suspect that \citet{shao:zhang:2010} applied the median-based test in the same manner as the CUSUM test, restricting the location of the potential change-point to the years 1952--1962, making the test largely resemble a two-sample test.
		
\begin{figure}[t]
\centering
\includegraphics[width=0.85\textwidth]{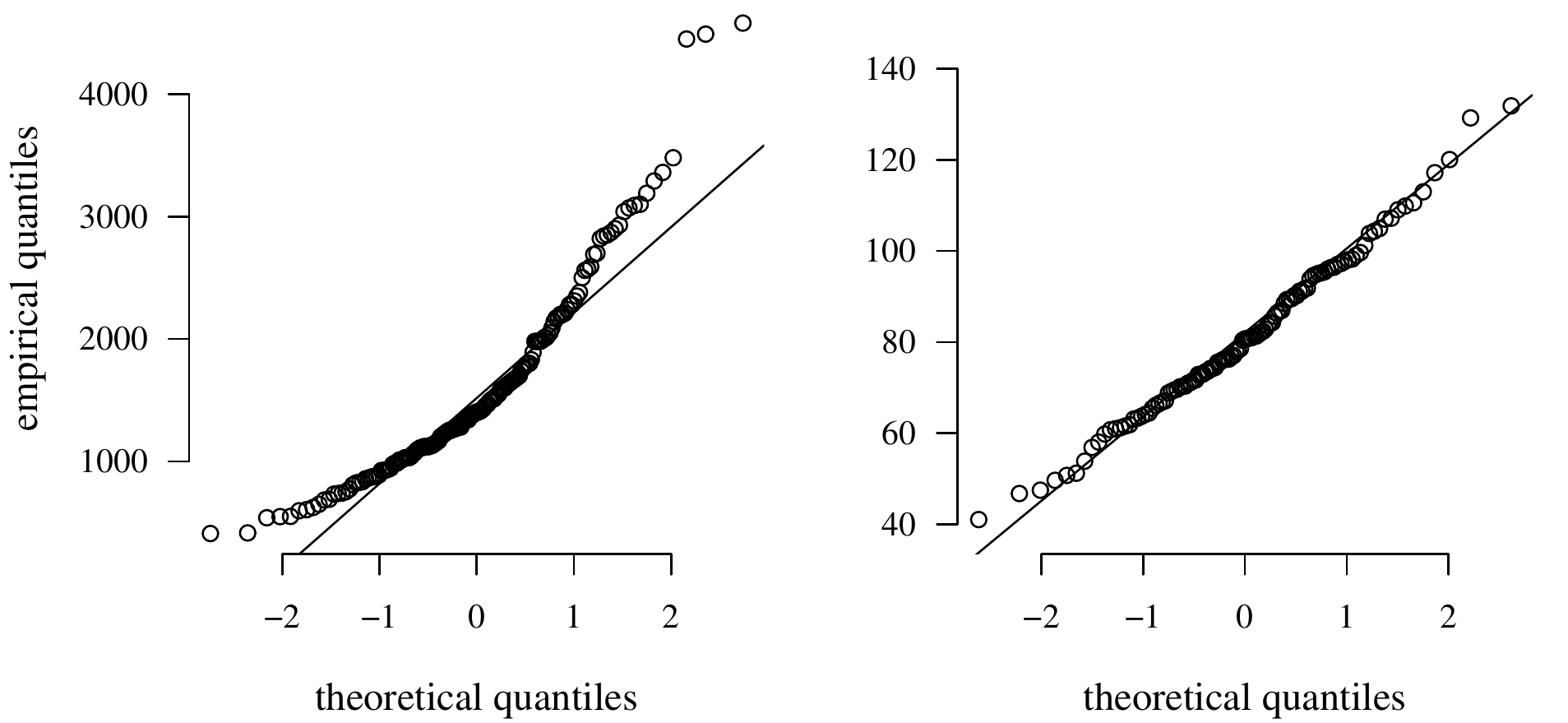}
\caption{Normal quantile plots for the River Elbe discharge data (left) and the Argentina rainfall data (right).}
\label{fig:qq}
\end{figure}


\section{Summary and discussion}

We have proved a functional limit theorem for the general $U$-quantile process for short-range dependent data. We have furthermore established the consistency of an HAC kernel estimator for the long-run variance. The results are formulated under very mild conditions on the data. We use near epoch dependence in probability (\PNED) on mixing sequences to capture the short-range dependence, which does do not imply any moment condition.

As an application of the theory, we examine the properties of a new change-point test for location. The test is of the plug-in type, obtained from the classical CUSUM test by replacing the mean by the Hodges--Lehmann estimator. It is demonstrated by simulations and also mediated by the two data examples that the Hodges--Lehmann test outperforms the CUSUM test at heavy-tailed data and significantly reduces the potential harm of gross errors, but essentially behaves as the CUSUM test under normality. We show that the Hodges--Lehmann estimator is clearly to be preferred over the median for this purpose. A drawback of the Hodges--Lehmann test is the higher computational cost, but this has become negligible with the use of computers. 

The problem of robust univariate location estimation is well studied with \citet{Huber1964} being one of the main contributions, and there are other robust estimators that might perform comparably to the Hodges--Lehmann estimator in this context. See, e.g.\ \citet[][Chapters 3 \& 4]{Huber2009} for an overview on robust location estimation. However, besides its good statistical properties, the Hodges--Lehmann estimator possesses an intriguing conceptual simplicity: there are no weight functions, trimming percentages, tuning constants, etc., to choose.
Furthermore, a thorough mathematical analysis of robust estimators generally tends to be elaborate, and the literature on functional limit theorems for such estimators is rather limited. \citet{jureckova:sen:1981a,jureckova:sen:1981b} are works in this direction, but we are not aware of any results for dependent data. 

A certain reservation towards the use of robust estimators in general stems from
the strong focus on moment characteristics as descriptive parameters of distributions. For instance, the mean is widely used to describe the central location, and any alternative location measure, such as the median or the Hodges--Lehmann estimator, coincides with the mean only under some restrictive assumptions on the data distribution (e.g.\ symmetry). This objection against the use of robust estimators is of much lesser legitimacy for two-sample or change-point tests. If we consider explicitly the change-point model described in the introduction, where the observations before and after the change-point differ only by a shift, but otherwise follow the same distribution, this shift is picked up equally by any proper, translation equivariant location measure, and one is hence free to make the choice solely based on the statistical properties of the estimators.


\appendix

\section{Proof of Theorem \ref{theo:1}}
\label{sec:appA}

All throughout the Appendix, we use $C$ as generic notation for a constant. Its value may change from line to line, but it is always independent of $n$ and all other indices involved in the respective statement. Further, we write $\left\|\cdot\right\|_{p}=\left(E|\cdot|^p\right)^{1/p}$ for the $L_p$-norm of a random variable. In Appendix \ref{sec:appA} we prove Theorem \ref{theo:1}. Appendix \ref{sec:appB} is devoted to the proof of Theorem \ref{theo:2}.

We start by gathering several important auxiliary results from the literature which are stated here without proof.
The following weak invariance principle for $U$-statistics is a variant of Theorem 2.5 of \citet{dehling:vogel:wendler:wied:2015:vs4} for bounded kernels.
\citet{dehling:vogel:wendler:wied:2015:vs4} state the invariance principle for unbounded kernels, assuming $(2+\delta)$-moments. The bounded case can be proved in the same way, so we omit the proof.

\begin{proposition}\label{theo1} Under Assumptions 1, 2, and 3, the $U$-statistic process
\begin{equation*}
\left(\frac{[ns]}{\sqrt{n}}\left(U_{[ns]}(U^{-1}(p))-p\right)\right)_{s\in[0,1]}
\end{equation*}
converges weakly in $D[0,1]$ to $\sigma W$, there $W$ is a standard Brownian motion, and
\begin{equation*}
\sigma^2= 4 \sum_{r=-\infty}^\infty \cov\left(h_1(X_0,U^{-1}(p)),h_1(X_r,U^{-1}(p))\right).
\end{equation*}
\end{proposition}
We will approximate $U$-quantiles by $U$-statistics and will make repeated use of  $U$-statistic results. 
Similarly to (\ref{eq:hoeffding}), we can define the Hoeffding decomposition of the kernel $g$, and define $g_1(x) = E g(x,Y) - E g(X,Y)$, where $X$, $Y \sim F$ are i.i.d.\ and $F$ is the marginal distribution of the process $(X_i)_{i \in \Z}$. The following lemma is the analogue of Lemma A.2 of \citet{dehling:vogel:wendler:wied:2015:vs4} for bounded kernels $g$.
\begin{lemma}\label{lem1} 
Let $(X_n)_{n\in\Z}$ be a stationary and $P$-near epoch dependent process on $(Z_n)_{n\in\Z}$ with approximating constants $a_l$ and non-increasing function $\phi$. Let further $g$ be a bounded, symmetric kernel satisfying the variation condition (Assumption \ref{ass:variation2}).
If there is a sequence of positive numbers $(s_l)_{l\in\N}$ such that $a_l\phi(s_l)=O(s_l)$, then the sequence $(g_1(X_n))_{n\in\Z}$ is $L_2$-NED on $(Z_n)_{n\in\Z}$, and the approximation constants satisfy $a_{l,2}=O(s_l^{1/2})$.
\end{lemma}
\begin{lemma}\label{lem6} Under Assumptions \ref{ass:ned} and \ref{ass:variation}, we have for any $t\in\R$
\begin{equation*}
\left\|\max\nolimits_{n\leq 2^k}\sum\nolimits_{1\leq i<j\leq n}h_2(X_i,X_j,t)\right\|_2\leq C 2^{\frac{5}{4}k}k,
\end{equation*}
and 
$\sum_{1\leq i<j\leq n}h_2(X_i,X_j,t)=O\left(n^{5/4}\log^2 (n)\right)$ 
almost surely.
\end{lemma}
This is Lemma B.6 of \citet{dehling:vogel:wendler:wied:2015:vs4}.

\begin{proposition}\label{pro1} Under Assumptions \ref{ass:ned} and \ref{ass:variation},  we have
$U_n(U^{-1}(p))-p = O(\sqrt{\log\log(n)/ n})$  almost surely.
\end{proposition}
\begin{proof}
We use the Hoeffding decomposition
\begin{equation*}
U_n(U^{-1}(p))-p=\frac{2}{n}\sum_{i=1}^nh_1(X_i,U^{-1}(p))+\frac{2}{n(n-1)}\sum_{1\leq i<j\leq n}h_2(X_i,X_j,U^{-1}(p)).
\end{equation*}
For the second summand, we can use Lemma \ref{lem6}. By Lemma \ref{lem1}, the sequence $(h_1(X_n,U^{-1}(p)))_{n\in\Z}$ is $L_2$-NED, so we can use the law of the iterated logarithm as in Theorem 8 of \citet{oodaira:yoshihara:1971} for first summand.
\end{proof}

To approximate the $U$-quantiles by $U$-statistics, we use the following generalized Bahadur representation.
\begin{proposition} \label{pro2} Under Assumptions \ref{ass:ned}, \ref{ass:variation}, and \ref{ass:smooth}, we have
\begin{enumerate}[(A)]
\item
$ \displaystyle 
  \sup_{\left|t-U^{-1}(p)\right|\leq C\sqrt{\frac{\log\log n}{n}}}\left|U_n(t)-U(t)-U_n\left(U^{-1}(p)\right)+p\right|=O\left(n^{-5/8}\right)$
\ \ and
\item
$ \displaystyle 
 R_n = U^{-1}_n\left(p\right)-U^{-1}(p)+\frac{U_n\left(U^{-1}(p)\right)-p}{u\left(U^{-1}(p)\right)}=O\left(n^{-5/8}\right)$
almost surely as $n\rightarrow\infty$.
\end{enumerate}
\end{proposition}
\begin{proof} 
To shorten notation, we abbreviate $U^{-1}(p)$ by $t_0$. Keep in mind that $U(t_0) = p$.

\emph{Part (A)}: 
We set $c_n = 2^{-5k/8}$ for $n = 2^{k-1}+1,\ldots,2^k$ and $k \in\N$. Note that $U(t)$ and $U_n(t)$ are non-decreasing, so for any $m\in\N$ and any $t\in[t_0+mc_n,t_0+(m+1)c_n]$ we have
\begin{multline*}
\left|U_n(t)-U_n\left(t_0\right)-U(t)+p\right|\\
\leq \max\big\{\left|U_n\left(t_0+m c_n\right)-U_n\left(t_0\right)-U(t)+p\right|, \\
\qquad \qquad \left|U_n\left(t_0+(m+1)c_n\right)-U_n\left(t_0\right)-U\left(t_0+(m+1)c_n\right)+p\right|\big\}\\
\leq \max\big\{\left|U_n\left(t_0+m c_n\right)-U_n\left(t_0\right)-U(t)+p\right|,\\
 \qquad \qquad \left|U_n\left(t_0+(m+1)c_n\right)-U_n\left(t_0\right)-U\left(t_0+(m+1)c_n\right)+p\right|\big\}\\
+\left|U\left(t_0+(m+1)c_n\right)-U\left(t_0+m c_n\right)\right|.
\end{multline*}
Using this inequality for all $t$ such that $|t-t_0|\leq C\sqrt{(\log k)/2^k}$, it follows that
\begin{align*}
			&\, \sup_{|t-t_0|\leq C\sqrt{\frac{\log k}{2^k}}}\left|U_n(t)-U_n\left(t_0\right)-U(t)+p\right|\\
\leq	&\, \max_{\left|m\right|\leq C2^{-5k/8}\log k}\left|U_n\left(t_0+(m+1)c_n\right)-U_n\left(t_0\right)-U\left(t_0+mc_n\right)+p\right|\\
			&+\max_{\left|m\right|\leq C2^{-k/8}\log k}\left|U\left(t_0+(m+1)c_n\right)-U\left(t_0+mc_n\right)\right|,
\end{align*}
and by Assumption \ref{ass:smooth} on the differentiability of the $U$-distribution function:
\begin{equation*}
\max_{\left|m\right|\leq C2^{-5k/8}\log k}\left|U\left(t_0+(m+1)c_n\right)-U\left(t_0+mc_n\right)\right|=O\left(c_n\right).
\end{equation*}
We use the Hoeffding decomposition and treat the linear part and the degenerate part separately:
\begin{multline*}
\max_{\left|m\right|\leq C2^{k/8}\log k}\left|U_n\left(t_0+(m+1)c_n\right)-U_n\left(t_0\right)-U\left(t_0+mc_n\right)+p\right|\\
\leq \max_{\left|m\right|\leq C2^{k/8}\log k}\left|\frac{2}{n}\sum_{i=1}^{n}h_{1}\left(X_{i},t_0+mc_n\right)-\frac{2}{n}\sum_{i=1}^{n}h_{1}\left(X_{i},t_0\right)\right|\\
+\max_{\left|m\right|\leq C2^{k/8}\log k}\left|\frac{2}{n(n-1)}\sum_{i=1}^{n}h_{2}\left(X_{i},X_j,t_0+mc_n\right)-\frac{2}{n(n-1)}\sum_{i=1}^{n}h_{2}\left(X_{i},X_j,t_0\right)\right|
\end{multline*}
The functions satisfying the variation condition (Assumption \ref{ass:variation}) form a vector space, so for $h_1$ the variation condition holds uniformly in some neighborhood of $t_0$. Furthermore, the sequence $(h_1(X_n,t_0))_{n\in\Z}$ is $L_2$-NED by Lemma \ref{lem1} and thus the approximation condition of \citet{Wendler2011} holds. Applying Theorem 1 of \citet{Wendler2011} to the function $g=h_1$, we obtain
\begin{multline*}
\max_{\left|m\right|\leq C2^{k/8}\log k}\left|\frac{2}{n}\sum_{i=1}^{n}h_{1}\left(X_{i},t_0+mc_n\right)-\frac{2}{n}\sum_{i=1}^{n}h_{1}\left(X_{i},t_0\right)\right|\\
\leq \sup_{|t-t_0|\leq C\sqrt{\frac{\log k}{2^k}}}\left|\frac{2}{n}\sum_{i=1}^{n}h_{1}\left(X_{i},t\right)-\frac{2}{n}\sum_{i=1}^{n}h_{1}\left(X_{i},t_0\right)\right|
=O\left(c_n\right)
\end{multline*}
almost surely. It remains to show that
\begin{equation} \label{eq:partA.ende}
	\max_{\left|m\right|\leq C2^{k/8}\log k}\left|Q_n\left(t_0+mc_n\right)-Q_n\left(t_0\right)\right|=O\left(n^2c_n\right)
\end{equation}
almost surely, where $Q_n(t) = \sum_{1\leq i<j\leq n}h_{2}\left(X_{i},X_{j},t\right)$. Recall that for any random variables $Y_1,\ldots,Y_m$, it holds $E\left(\max_{i=1,\ldots,m}|Y_i|\right)^2\leq\sum_{i=1}^m EY_i^2$ and therefore
\[
	E \left(
		\max_{2^{k-1}< n\leq 2^k}\max_{\left|m\right|\leq C2^{k/8}\log k}\frac{2}{2^{2k}c_n}
		\left|Q_n\left(t_0+c_n m\right)-Q_n\left(t_0\right)\right|
	\right)^2
\]\[
	\leq 
	\frac{2}{2^{4k} (2^{-5k/8})^2 }\sum_{\left|m\right|\leq C2^{k/8}\log k}
	 E \left(\max_{2^{k-1}< n\leq 2^k}\left|Q_n\left(t_0 + c_n m\right)-Q_n\left(t_0\right)\right|\right)^2,
\]
where we have used that $c_n = 2^{-5k/8}$ for $n = 2^{k-1} + 1, \ldots, 2^k$. The right-hand side is further bounded by
\[
	 \frac{4}{2^{11k/4}}
			\sum_{\left|m\right|\leq C2^{k/8}\log k}
			\! \! E\!\left(\max_{2^{k-1}< n\leq 2^k}\left|Q_n\left(t_0+c_n m\right)\right|\right)^2 
	\leq \frac{ C 2^{k/8}}{2^{11k/4}} \log(k) 2^{\frac{5}{2}k}k^2
	= C 2^{-\frac{k}{8}}k^2\log k,
\]
where we have applied Lemma \ref{lem6}. Using the Markov inequality, we conclude that
\begin{multline*}
 \sum_{k=1}^\infty P\left(\max_{2^{k-1}< n\leq 2^k}\max_{\left|m\right|\leq C2^{k/8}\log k}\frac{2}{2^{2k}c_n}\left|Q_n\left(t_0+c_n m\right)-Q_n\left(t_0\right)\right|>\epsilon\right)\\
\leq \sum_{k=1}^\infty\frac{1}{\epsilon^2}E\left(\max_{2^{k-1}< n\leq 2^k}\max_{\left|m\right|\leq C2^{k/8}\log k}\frac{2}{2^{2k}c_n}\left|Q_n\left(t_0+c_n m\right)-Q_n\left(t_0\right)\right|\right)^2\\
\leq C\sum_{k=1}^\infty 2^{-\frac{k}{8}}k^2\log k<\infty,
\end{multline*}
and with the Borel--Cantelli lemma (\ref{eq:partA.ende}) follows, and hence Part (A) is proved.
%
%
%
%
%
%
%
%

\emph{Part (B):} 
Without loss of generality, let $u\left(t_0\right)=1$, otherwise replace $h(x,y,t)$ by $h(x,y,\frac{t}{u(t_0)})$. We represent $R_n$ as $Z_n\left(p-U_n\left(t_0\right)\right)$ with
\begin{equation*}
 Z_n\left(x\right)=\left(U_n\left(\cdot+t_0\right)-U_n\left(t_0\right)\right)^{-1}\left(x\right)-x=U_n^{-1}\left(x+U_n\left(t_0\right)\right)-x-t_0,
\end{equation*}
where $\left(U_n\left(\cdot+t_0\right)-U_n\left(t_0\right)\right)^{-1}$ is the inverse function of $x \mapsto U_n\left(x+t_0\right)-U_n\left(t_0\right)$. By Proposition \ref{pro1}, we have
$ \limsup_{n\rightarrow\infty}	\pm	\sqrt{n(\log\log n)^{-1}} \left(U_n\left(t_0\right)-p\right) = C$.
By Assumption \ref{ass:smooth} and Part (A), we have
\begin{eqnarray*}
	\sup_{\left|x\right|\leq C\sqrt{(\log \log n) / n}}\left|Z_n(x)\right|
	& = 	& \sup_{\left|x\right|\leq C\sqrt{(\log \log n) / n}}\left|U_n\left(x+t_0\right)-U_n\left(t_0\right)-x\right| \\
  & \leq&  \sup_{\left|x\right|\leq C\sqrt{(\log \log n) / n}}\left|U_n\left(x+t_0\right)-U\left(x+t_0\right)-U_n\left(t_0\right)+p\right|\\
  &     & + \sup_{\left|x\right|\leq C\sqrt{(\log \log n) / n}}\left|U\left(x+t_0\right)-p-x\right|=O\left(c_n\right).
\end{eqnarray*}
Then by Theorem 1 of \citet{Vervaat1972}, 
$ \left|R_n\right|\leq\sup_{\left|x\right|\leq C\sqrt{(\log \log n) / n}}\left|Z_n\left(x\right)\right|=O\left(c_n\right)$,
so Part (B) of Proposition \ref{pro2} is proved.
\end{proof}

We are now ready to prove Theorem \ref{theo:1}.
\begin{proof}[Proof of Theorem \ref{theo:1}] We write
\begin{equation*}
		\frac{[ns]}{\sqrt{n}}\left(U_{[ns]}^{-1}(p)-U^{-1}(p)\right)
		=\frac{[ns]}{\sqrt{n}}\frac{p-U_n\left(U^{-1}(p)\right)}{u\left(U^{-1}(p)\right)}+\frac{[ns]}{\sqrt{n}}R_n.
\end{equation*}
where $R_n$ is as in Proposition \ref{pro2}.
By Proposition \ref{theo1},
\begin{equation*}
\left(\frac{[ns]}{\sqrt{n}}\bigg(\frac{p-U_n\left(U^{-1}(p)\right)}{u\left(U^{-1}(p)\right)}\bigg)\right)_{s\in[0,1]}
\end{equation*}
converges weakly in $D[0,1]$ to $\sigma W$, there $W$ is a standard Brownian motion and the variance is given by
\begin{equation*}
\sigma^2=\frac{4}{u^2(U^{-1}(p))}\sum_{k=-\infty}^\infty \cov\left(h_1(X_0,U^{-1}(p)),h_1(X_k,U^{-1}(p))\right)
\end{equation*}
By Proposition \ref{pro2}, we have $n^{5/8}R_n\rightarrow 0$  almost surely
 and thus $|nR_n|\leq C n^{3/8}$ almost surely. Consequently
\begin{equation*}
\sup_{s\in[0,1]}\frac{[ns]}{\sqrt{n}}|R_{[ns]}|=\frac{1}{\sqrt{n}}\max_{k\leq n}k|R_{k}|\leq\frac{1}{\sqrt{n}}Cn^{3/8}\rightarrow 0
\end{equation*}
almost surely, and Slutsky's theorem completes the proof.
\end{proof}

%
%
%
%
%
%
%
%
%
%
%
%
\section{Proof of Theorem \ref{theo:2}}
\label{sec:appB}

The proof of Theorem \ref{theo:2} consists of two main steps: showing the convergence of the density estimator $\hat{u}_n^2(U^{-1}_n(p))$ to $u^2(U^{-1}(p))$ and showing the convergence of the cumulative autocovariance part. The former is the content of Lemma \ref{lem13}. The following Lemma \ref{lem12} is an essential tool for the latter step. 
\begin{lemma}\label{lem12} Under Assumptions \ref{ass:ned}, \ref{ass:variation}, \ref{ass:smooth}, and \ref{ass:hac} we have 
\[ 
\sum_{r=-(n-1)}^{(n-1)}\frac{1}{n}\sum_{i=1}^{n-|r|}\left(\hat{h}_1(X_i,t_n)\hat{h}_1(X_{i+|r|},t_n)-\hat{h}_1(X_i,t_0)\hat{h}_1(X_{i+|r|},t_0)\right)W\left(\frac{|r|}{b_n}\right)
\] 
converges to 0 in probability as $n\rightarrow\infty$, where we have abbreviated $t_0 = U^{-1}(p)$ and $t_n = U^{-1}_n(p)$.
\end{lemma}

\begin{proof} We first have a look at the covariance estimator for a fixed lag $r$. We will use the facts that $|h|\leq 1$, $|\hat{h}|\leq 1$ and that $h$ is non-decreasing in the third argument.
\[ \textstyle
	\left|\frac{1}{n}\sum_{i=1}^{n-r}\left(\hat{h}_1(X_i,t_n)\hat{h}_1(X_{i+r},U^{-1}_n(p))-\hat{h}_1(X_i,t_0)\hat{h}_1(X_{i+r},t_0)\right)\right|
\] \[ \textstyle
	\leq \left|\frac{1}{n}\sum_{i=1}^{n-r}\left(\hat{h}_1(X_i,t_n)-\hat{h}_1(X_i,t_0)\right)\hat{h}_1(X_{i+r},t_n)\right|
\] \[ \textstyle
	\quad +\left|\frac{1}{n}\sum_{i=1}^{n-r}\hat{h}_1(X_i,t_0)\left(\hat{h}_1(X_{i+r},t_n)-\hat{h}_1(X_{i+r},t_0)\right)\right| 
\] \[ \textstyle
	\leq\left|\frac{1}{n}\sum_{i=1}^{n-r}\left(\frac{1}{n}\sum_{j=1}^nh(X_i,X_j,t_n)-\frac{1}{n}\sum_{j=1}^nh(X_i,X_j,t_0)\right)\hat{h}_1(X_{i+r},t_n)\right|
\] \[ \textstyle
 + \left|\frac{1}{n}\sum_{i=1}^{n-r}
		\left(\frac{1}{n^2}\sum_{j_1,j_2=1}^nh(X_{j_1},X_{j_2},t_n)-\frac{1}{n^2}\sum_{j_1,j_2=1}^nh(X_{j_1},X_{j_2},t_0)\right)\hat{h}_1(X_{i+r},t_n)
	 \right| 
\] \[ \textstyle
	\quad + \left|
		\frac{1}{n}\sum_{i=1}^{n-r}\hat{h}_1(X_i,t_0)
		\left(\frac{1}{n}\sum_{j=1}^nh(X_{i+r},X_j,t_n)-\frac{1}{n}\sum_{j=1}^nh(X_{i+r},X_j,t_0)\right)
	\right|
\] \[ \textstyle
	\quad + \left|
		\frac{1}{n}\sum_{i=1}^{n-r} \hat{h}_1(X_i,t_0)
		\left(\frac{1}{n^2}\sum_{j_1,j_2=1}^nh(X_{j_1},X_{j_2},t_n)-\frac{1}{n^2}\sum_{j_1,j_2=1}^nh(X_{j_1},X_{j_2},t_0)\right)
	\right| 
\] \[ \textstyle
	\leq\left|\frac{1}{n^2}\sum_{i=1}^{n-r}\sum_{j=1}^n\left(h(X_i,X_j,t_n)-h(X_i,X_j,t_0)\right)\right|
\] \[ \textstyle
	\quad +\left|\frac{n-r}{n}\frac{1}{n^2}\sum_{j_1,j_2=1}^n\left(h(X_{j_1},X_{j_2},t_n)-h(X_{j_1},X_{j_2},t_0)\right)\right|
\] \[ \textstyle
	\quad +\left|\frac{1}{n^2}\sum_{i=1}^{n-r}\sum_{j=1}^n\left(h(X_{i+r},X_j,t_n)-h(X_{i+r},X_j,t_0)\right)\right|
\] \[ \textstyle
	\quad +\left|\frac{n-r}{n}\frac{1}{n^2}\sum_{j_1,j_2=1}^n\left(h(X_{j_1},X_{j_2},t_n)-h(X_{j_1},X_{j_2},t_0)\right)\right| 
\] \[ \textstyle
	\leq 4\left|\frac{1}{n^2}\sum_{j_1,j_2=1}^n\left(h(X_{j_1},X_{j_2},t_n)-h(X_{j_1},X_{j_2},t_0)\right)\right|
 \] \[ \textstyle
	\leq 4\left|\frac{1}{n^2}\sum_{j_1,j_2=1}^n\left(h(X_{j_1},X_{j_2},t_n)-h(X_{j_1},X_{j_2},t_0)-U(t_n)+p\right)\right| 
	+ 4\left|U(t_n)-p\right|
\]
First note that the right-hand side of this chain of inequalities does not depend on $r$. By Propositions \ref{pro1} and \ref{pro2}, 
$|t_n - t_0| = O\left( \sqrt{\log\log(n)/n }\right)$ almost surely. So we can conclude with the help of Proposition \ref{pro2} that
\[ \textstyle 
	\left|\frac{1}{n^2}\sum_{j_1,j_2=1}^n\left(h(X_{j_1},X_{j_2},t_n)-h(X_{j_1},X_{j_2},t_0)-U(t_n)+p\right)\right|
\]\[ \textstyle
	\leq \left|U_n(t_n)-U_n(t_0)-U(t_n)+p\right|
	+\left|\frac{1}{n^2}\sum_{j=1}^n\left(h(X_{j},X_{j},t_n)-h(X_{j},X_{j},t_0)-U(t_n)+p\right)\right|
\] \[ 
	\leq \sup_{\left|t-t_0\right|\leq C\sqrt{\frac{\log\log n}{n}}}\left|U_n(t)-U(t)-U_n\left(t_0\right)+p\right| + C/n = O(n^{-5/8})
\]
almost surely. From Assumption \ref{ass:smooth} and Theorem \ref{theo:1}, we conclude that $|U(t_n)-p|\leq C(t_n-t_0) = O_P(n^{-1/2})$,
and finally arrive at
\begin{multline*}
\sum_{r=-(n-1)}^{n-1}\frac{1}{n}\sum_{i=1}^{n-|r|}\left(\hat{h}_1(X_i,t_n)\hat{h}_1(X_{i+|r|},t_n)-\hat{h}_1(X_i,t_0)\hat{h}_1(X_{i+|r|},t_0)\right)W(|r|/b_n)\\
\leq C\frac{1}{\sqrt{n}} \sum_{r=-n}^nW(|r|/b_n)\rightarrow 0
\end{multline*}
in probability as $n\rightarrow\infty$. The proof is complete.
\end{proof}

\begin{lemma}\label{lem13} Under Assumptions \ref{ass:ned}, \ref{ass:smooth}, \ref{ass:density}, and \ref{ass:variation2},
\begin{equation*}
	\hat{u}_n =\frac{2}{n(n-1)d_n}\sum_{1\leq i<j\leq n}K((g(X_i,X_j)-U_n^{-1}(p))/d_n)\rightarrow u(U^{-1}(p))
\end{equation*}
in probability as $n\rightarrow\infty$.
\end{lemma}

\begin{proof} 
We introduce an upper kernel $K_{u,n}$ and a lower kernel $K_{l,n}$ by
\begin{equation*}
	K_{u,n}(t)	=\sup_{t':\ |t-t'|\leq\sqrt{\frac{\log n}{n}}}K(t) 
	\ \ \ \text{and} \ \ \ 
	K_{l,n}(t)  = \inf_{t':\ |t-t'|\leq\sqrt{\frac{\log n}{n}}}K(t),
\end{equation*}
and further an upper estimate $\hat{u}_{u,n}$ and a lower estimate $\hat{u}_{l,n}$ by
\begin{align*}
	\hat{u}_{u,n} & =\frac{2}{n(n-1)d_n}\sum_{1\leq i<j\leq n}K_{u,n}\left(\frac{g(X_i,X_j)-U^{-1}(p)}{d_n}\right),\\
	\hat{u}_{l,n} & =\frac{2}{n(n-1)d_n}\sum_{1\leq i<j\leq n}K_{l,n}\left(\frac{g(X_i,X_j)-U^{-1}(p)}{d_n}\right).
\end{align*}
Since $|U_n^{-1}(p)-U^{-1}(p)|=O(\sqrt{\log\log(n)/n})$ almost surely (Propositions \ref{pro1} and \ref{pro2}), we have almost surely $\hat{u}_{l,n}\leq \hat{u}_n\leq \hat{u}_{u,n}$ for all but a finite number of $n$. Hence it suffices to show that $\hat{u}_{u,n}\rightarrow u(U^{-1}(p))$ and $\hat{u}_{l,n}\rightarrow u(U^{-1}(p))$ in probability as $n\rightarrow\infty$. We will focus on $\hat{u}_{u,n}$, as the proof for $\hat{u}_{l,n}$ is analogous. Note that $\hat{u}_n$ is a $U$-statistic with symmetric kernel $k_n(x,y)=K_{u,n}((g(x,y)-U^{-1}(p))/d)$ depending on $n$. We use the Hoeffding decomposition
\[
\tilde{u}_n =Ek_n(X,Y), \ \  k_{1,n}(x) = Ek_n(x,X_i)-\tilde{u}_n, \ \ k_{2,n}(x,y)=k_n(x,y)-k_{1,n}(x)-k_{1,n}(y)-\tilde{u}_n,
\]
where $X$, $Y$ are independent with the same distribution as $X_0$. We obtain
\be \label{eq:decomposition}
	\hat{u}_{u,n}=\tilde{u}_n+\frac{2}{n}\sum_{i=1}^nk_{1,n}(X_i)+\frac{2}{n(n-1)}\sum_{1\leq i<j\leq n}k_{2,n}(X_i,Y_i).
\ee
We treat the three summands on the right-hand side separately.
%
%
%
%
%
%
%
%
%
%
By our assumptions, $K$ has a bounded support, so let $K(x)=0$ for $|x|>M$. Because the density $u$ is continuous and $K$ integrates to 1, we can conclude that
\begin{multline*}
\tilde{u}_n-u(U^{-1}(p))=\int \frac{1}{d_n}K_{u,n}\left(\frac{x-U^{-1}(p)}{d_n}\right)u(x)dx-u(U^{-1}(p))\\
=\int K_{u,n}(x)u(xd_n+U^{-1}(p))dx-u(U^{-1}(p+))\\
\leq\int K_{u,n}(x)\left|u(xd_n+U^{-1}(p))-u(U^{-1}(p))\right|dx+\left(\int K_{u,n}(x)dx-1\right)\\
\leq 2\left(Md_n+\sqrt{\frac{\log n}{n}}\right)\sup_{|x|\leq Md_n+\sqrt{\frac{\log n}{n}}}\left|u(xd_n+U^{-1}(p))-u(U^{-1}(p))\right|\sup_{x\in\R}K(x),
\end{multline*}
which converges to 0 as $n \to \infty$ since $d_n\rightarrow 0$. 
%
%
%
%
%
%
%
%
%
%
To prove the convergence of the second and third summand in the Hoeffding decomposition (\ref{eq:decomposition}), we first gather some properties of the sequence $k_{n}$. Kernel $K$ is Lipschitz continuous for some constant $L_1$, that is $|K(x)-K(y)|\leq L_1|x-y|$, hence the mapping $x\rightarrow\frac{1}{d}K(\frac{x}{d})$ is Lipschitz continuous with constant $L_1/d^2$, and $k_n(x,y)=\frac{1}{d}K(\frac{g(x,y)}{d})$ satisfies the variation condition (Assumption \ref{ass:variation}) with constant $L'=C d^{-4}$. Furthermore, $k_{n}(x,y)\leq M'=C\frac{1}{d}$ and $E|k_{n}(X,Y)|\leq C$ for independent $X$, $Y$ and thus $E k_{1,n}^2\leq C\frac{1}{d}$. 
By the proof of Lemma \ref{lem1} we find that $(k_{1,n}(X_i))_{i\in\Z}$ is $L_2$-near epoch dependent with approximation constants 
$a'_l=C/d_n^2l^{-3}$. As in the proof of Lemma C.1 of \citet{dehling:vogel:wendler:wied:2015:vs4}, we have that
\[
	\left|Ek_{1,n}(X_i)k_{1,n}(X_{i+k})\right|
\leq 10\left\|E(k_{1,n}(X_{i+k})|\mathcal{G}_{i+k-l}^{i+k+l})\right\|_{2+\delta}^2\beta^{\frac{\delta}{2+\delta}}_{k-2l}
\]
\[
	\quad +\,2\left\|k_{1,n}(X_i)\right\|_2\left\|k_{1,n}(X_{i+k})-E(k_{1,n}(X_{i+k})|\mathcal{G}_{i+k-l}^{i+k+l})\right\|_2
\leq C\frac{1}{d_n^2}\beta_{l}+C\frac{1}{\sqrt{d_n}}a'_l,
\]
where $\mathcal{G}_i^j$ denotes the $\sigma$-field generated by $Z_i,\ldots, Z_j$,
so we obtain by stationarity that
\begin{equation*}
	E\left(\frac{2}{n}\sum_{i=1}^nk_{1,n}(X_i)\right)^2\leq\frac{4}{n}\sum_{i=1}^\infty\left|E\left( k_{1,n}(X_1)k_{1,n}(X_i)\right)\right|\leq C\frac{1}{nd_n^{5/2}}\sum_{i=1}^\infty((3/i)^{3}+\beta_{i})
\end{equation*}
converges to 0 since $n d_n^{5/2}\rightarrow\infty$. So the second summand of (\ref{eq:decomposition}) converges to 0. 
%
%
%
%
%
%
%
%
%
%
For the degenerate part, we use that $k_{2,n}(x,y)$ is a degenerate kernel bounded by $C/d_n$, so we can prove similarly to Lemma B.2 of \citet{dehling:vogel:wendler:wied:2015:vs4} that
\begin{equation*}
\left\|k_{2,n}(X_i,X_{i+k+2l})-k_{2,n}(X_{i,l},X_{i+k+2l,l})\right\|_2\leq C(\sqrt{L'\epsilon}+M'a_l^{\frac{\delta}{2+\delta}}\phi^{\frac{\delta}{2+\delta}}(\epsilon)+M'\beta_k),
\end{equation*}
where we write $X_{i,l}$ short for $f_l(Z_{i-l},\ldots,Z_{i+l})$,
and can conclude that
\[
	\left\|\frac{2}{n(n-1)}\sum_{1\leq i<j\leq n}\left(k_{2,n}(X_i,X_j)-k_{2,n}(X_{i,l},X_{i+k+2l,l})\right)\right\|_2
\]\[
	\leq Cn^{-3/4}(\sqrt{ML'}+M')\leq C(d_n^{8/3} n )^{-3/4}\rightarrow 0
\]
by our assumptions on $d_n$. Similarly (compare Lemma B.4 of \citet{dehling:vogel:wendler:wied:2015:vs4}) we get
\begin{equation*}
	\left\| \frac{2}{n(n-1)}\sum_{1\leq i<j\leq n}k_{2,l,n}(X_{i,l},X_{j,l})-k_{2,n}(X_{i,l},X_{j,l})\right\|_2\leq Cn^{-3/4}\left(\sqrt{ML'}+M'\right)\rightarrow 0,
\end{equation*}
where $k_{2,l,n}$ is defined by the Hoeffding decomposition of $k_n$ with respect to the distribution of $X_{0,l}$. 
Finally, as in Lemma B.5 of \citet{dehling:vogel:wendler:wied:2015:vs4},
\begin{equation*}
	\left|Ek_{2,l,n}(X_{i_1,l},X_{i_2,l})k_{2,l,n}(X_{i_3,l},X_{i_4,l})\right|\leq C(M')^2\beta_{m-l},
\end{equation*}
with $m=\max\left\{i_{(2)}-i_{(1)},i_{(4)}-i_{(3)}\right\}$, where $i_{(1)},\ldots,i_{(4)}$ are the ordered indices $i_1,i_2,i_3,i_4$, and thus
\begin{equation*}
	E\left(\sum_ {1\leq i<j\leq n}k_{2,l,n}(X_{i,l},X_{j,l})\right)^2\leq Cn^{-2}(M')^2l^2\rightarrow0
\end{equation*}
for $l=\lfloor n^{1/4}\rfloor$. We convergence of $\frac{2}{n(n-1)}\sum_{1\leq i<j\leq n}k_{2,n}(X_i,Y_j)$ then follows along the lines of the proof of Lemma \ref{lem6} (Lemma B.6 of \citet{dehling:vogel:wendler:wied:2015:vs4}), and hence $\hat{u}_{u,n}$ converges to $u(U^{-1}(p))$, and the proof is complete.
\end{proof}

\begin{proof}[Proof of Theorem \ref{theo:2}] We can rewrite the variance estimator $\hat{\sigma}_p^2$ as
\[
	\hat{\sigma}^2_p =\frac{4}{\hat{u}^2_n}\sum_{r=-(n-1)}^{n-1}W\Big(\frac{|r|}{b_n}\Big) \frac{1}{n}\sum_{i=1}^{n-|r|}\hat{h}_1(X_i,t_n)\hat{h}_1(X_{i+|r|},t_n)
\]\[
	= \frac{4}{\hat{u}^2_n}\sum_{r=-(n-1)}^{n-1}W\Big(\frac{|r|}{b_n}\Big)\frac{1}{n}\sum_{i=1}^{n-|r|}h_1(X_i,t_0)h_1(X_{i+|r|},t_0)
\] \[
	+\frac{4}{\hat{u}^2_n}\sum_{r=-(n-1)}^{n-1}\frac{1}{n}\sum_{i=1}^{n-|r|}
	\left(h_1(X_i,t_0)h_1(X_{i+|r|},t_0)-\hat{h}_1(X_i,t_0)\hat{h}_1(X_{i+|r|},t_0)\right)
	W\Big(\frac{|r|}{b_n}\Big)
\] \[
+\frac{4}{\hat{u}^2_n}\sum_{r=-(n-1)}^{n-1}\frac{1}{n}\sum_{i=1}^{n-|r|}
	\left(\hat{h}_1(X_i,t_0)\hat{h}_1(X_{i+|r|},t_0)-\hat{h}_1(X_i,t_n)\hat{h}_1(X_{i+|r|},t_n)\right)
	W\Big(\frac{|r|}{b_n}\Big).
\]
By Lemma \ref{lem13}, the density estimator $\hat{u}_n$ converges to $u$. Hence the first summand converges to $\sigma^2_p$  by Theorem 2.1 of \citet{dejong:2000} and Slutsky's theorem.
The second and the third summand converge to 0 by Lemma C.3 of \citet{dehling:vogel:wendler:wied:2015:vs4} and Lemma \ref{lem12}, respectively.
\end{proof}


\section*{Acknowledgement}
The research was supported by the DFG Sonderforschungsbereich 823 (Collaborative Research Center) {\em Statistik nichtlinearer dynamischer Prozesse}.
The authors thank Svenja Fischer and Wei Biao Wu for providing the river Elbe discharge data set and the Argentina rainfall data set, respectively. We are very grateful to the anonymous referee for the comments, which have helped to improve and clarify this manuscript.

\end{document}